\documentclass{amsart}

\usepackage{amsmath}
\usepackage{amssymb}
\usepackage{amsthm}

\usepackage[colorlinks,linkcolor=blue]{hyperref}

\usepackage[margin=1.0in]{geometry}               
\usepackage[color=green!40]{todonotes}
\usepackage{comment}

\usepackage{tikz}
\tikzstyle{ball} = [circle,shading=ball, ball color=black,
    minimum size=1mm,inner sep=1.3pt]
\tikzstyle{miniball} = [circle,shading=ball, ball color=black,
    minimum size=1mm,inner sep=0.5pt]
\usepackage{tikz-cd}

\usepackage{enumitem}

\usepackage{systeme}

\usepackage{needspace}  

\usepackage{url}

\newtheorem{thm}{Theorem}[section]
\newtheorem{lemma}[thm]{Lemma}

\newtheorem{cor}[thm]{Corollary}
\newtheorem{prop}[thm]{Proposition}
\newtheorem{property}[thm]{Property}

\newtheorem{Definition}[thm]{Definition}
\newenvironment{defn}
  {\begin{Definition}\rm}{\end{Definition}}

\newtheorem{Example}[thm]{Example}
\newenvironment{example}
  {\begin{Example}\rm}{\end{Example}}

\newtheorem{Remark}[thm]{Remark}
\newenvironment{remark}
  {\begin{Remark}\rm}{\end{Remark}}

\DeclareMathOperator{\Prin}{Prin}
\DeclareMathOperator{\Pic}{Pic}
\DeclareMathOperator{\Jac}{Jac}
\DeclareMathOperator{\Div}{Div}
\DeclareMathOperator{\smalldiv}{div}

\DeclareMathOperator{\diag}{diag}

\DeclareMathOperator{\im}{im}
\DeclareMathOperator{\cok}{cok}

\DeclareMathOperator{\ord}{ord}
\DeclareMathOperator{\Span}{Span}
\DeclareMathOperator{\GL}{GL}
\DeclareMathOperator{\Hom}{Hom}

\DeclareMathOperator{\code}{code}

\newcommand{\primary}{\mathcal{P}}
\newcommand{\secondary}{\mathcal{S}}

\newcommand{\I}{\mathbf{i}\mkern1mu}
\newcommand{\Z}{\mathbb{Z}}
\newcommand{\Q}{\mathbb{Q}}
\newcommand{\R}{\mathbb{R}}
\newcommand{\C}{\mathbb{C}}
\newcommand{\N}{\mathbb{N}}
\newcommand{\E}{\mathbb{E}}
\newcommand{\tL}{\tilde{L}}

\newcommand{\tl}{\ell}
\newcommand{\K}{\mathcal{K}}

\newcommand{\wA}{\widehat{A}}
\newcommand{\tK}{\widetilde{\mathcal{K}}}
\newcommand{\polyring}{\C[\mathbf{x}]}
\newcommand{\polyringy}{\C[\mathbf{y}]}
\newcommand{\neck}{\mathcal{N}}

\title{Enumerating Linear Systems on Graphs}
\author{Sarah Brauner}
\address{University of Minnesota, Minneapolis, MN}
\email{braun622@umn.edu}
\author{Forrest Glebe}
\address{Purdue University, West Lafayette, IN}
\email{fglebe@purdue.edu}
\author{David Perkinson}
\address{Reed College, Portland, OR}
\email{davidp@reed.edu}



\begin{document}

\begin{abstract} The divisor theory of graphs views a finite connected graph $G$
  as a discrete version of a Riemann surface.  Divisors on~$G$ are formal
  integral combinations of the vertices of~$G$, and linear equivalence of
  divisors is determined by the discrete Laplacian operator for~$G$.  As in the
  case of Riemann surfaces, we are interested in the complete linear
  system~$|D|$ of a divisor~$D$---the collection of nonnegative divisors
  linearly equivalent to~$D$.  Unlike the case of Riemann surfaces, the complete
  linear system of a divisor on a graph is always finite.  We compute generating
  functions encoding the sizes of all complete linear systems on~$G$ and
  interpret our results in terms of polyhedra associated with divisors and in
  terms of the invariant theory of the (dual of the) Jacobian group of~$G$.
  If~$G$ is a cycle graph, our results lead to a bijection between complete
  linear systems and binary necklaces.  Our results also apply to a model in
  which the Laplacian is replaced by an invertible,
  integral~$M$-matrix.
\end{abstract}

\maketitle

\section{Introduction}

Let~$G$ be a finite, connected, undirected graph with vertex set~$V$.  The
divisor theory of graphs uses the graph Laplacian to view~$G$ as a discrete
analogue of a Riemann surface.  As a reference, the reader should consult the
seminal paper by Baker and Norine (\cite{Baker}), a main result of which is the
Riemann-Roch theorem for graphs.  That work is related to a broader circle of
ideas that includes chip-firing on graphs (\cite{Bjorner}), the arithmetical groups
of Lorenzini (\cite{Lorenzini}), the abelian sandpile model (\cite{BTW},
\cite{Dhar}), and parking functions in combinatorics (\cite{Postnikov}).  For
general textbooks, including many references, see (\cite{Corry} and
\cite{Klivans}).  The papers~\cite{Holroyd} and~\cite{Bacher} are also recommended.

Precise definitions follow in Section~\ref{sect: prelims}, but for the purposes
of this introduction, it is useful to think of divisor theory on graphs in terms
of the {\em dollar game} introduced in \cite{Baker}.  By definition, a {\em
divisor}~$D$ is an element of ~$\Z V$, the free abelian group on the vertices
of~$G$.  Think of~$D$ as an assignment of~$D(v)$ dollars to each vertex~$v$.  If
the integer~$D(v)$ is negative, then~$v$ is in debt.  The net amount of money on
the graph is the {\em degree},~$\deg(D):=\sum_{v\in V}D(v)$, of~$D$.  A {\em
lending} move (or {\em firing}) by a vertex~$v$ consists of~$v$ giving one
dollar to each of its neighbors and losing the corresponding amount itself.  A
{\em borrowing} move is the opposite, in which~$v$ takes a dollar from each of
its neighbors.  Two divisors are {\em linearly equivalent} if one may be
transformed into the other by a sequence of lending and borrowing moves.  The
{\em Picard group},~$\Pic(G)$, is the group of divisors modulo linear
equivalence.  Since lending and borrowing conserve total wealth,~$\Pic(G)$
is graded by degree.  Its degree~$0$ part is a finite group called the {\em
Jacobian group},~$\Jac(G)$, and there is an
isomorphism~$\Pic(G)\simeq\Z\oplus\Jac(G)$, depending on the choice of a vertex
(cf.~\eqref{eqn: Pic(G)}).

The point of the dollar game is for the vertices to cooperate and, through a
sequence of lending and borrowing moves, reach a state in which no vertex is in
debt.  If this is possible, the effect is to transform the divisor~$D$ into a
new, linearly equivalent divisor ~$E$ that is debt-free, i.e.,  such
that~$E(v)\geq 0$ for all vertices~$v$.  Such a debt-free divisor is said to be
{\em effective}.  The {\em complete linear system} of a divisor~$D$,
denoted~$|D|$, is the set of all effective divisors linearly equivalent to~$D$.
In other words,~$|D|$ is the set of all winning states for the dollar game
starting with the initial distribution of wealth prescribed by~$D$.

The purpose of this paper is to answer the question: What is the cardinality of
the complete linear system~$|D|$ for each divisor~$D$?  In other words, how
many winning states are there for each dollar game on~$G$? 
This question was proposed by Haase, Musiker, and Yu at the end of their study
\cite{Haase} of linear systems on tropical curves. 
We know of no systematic study of this question prior to the work we present
here.  
To answer it, we first use the
isomorphism~$\Pic(G)\simeq\Z\oplus\Jac(G)$ to partition the collection of all
effective divisors on~$G$ into sets~$\E_{[D]}$, one for each~$[D]\in\Jac(G)$.
Let~$\lambda_{[D]}(k)$ be the number of divisors in~$\E_{[D]}$ with degree~$k$,
and let~$\Lambda_{[D]}(z):=\sum_{k\geq0}\lambda_{[D]}(k)z^k$ be its generating
function.  Our aim, then, is to understand the structure of the~$\E_{[D]}$ and
use it to find closed expressions for~$\Lambda_{[D]}(z)$ for
each~$[D]\in\Jac(G)$. The following is an outline of our results:
\smallskip

\noindent$\bullet$ Section~\ref{sect: primsec} shows that each effective
divisor has a decomposition into a sum of {\em primary} and {\em secondary}
divisors for~$G$ and then uses this idea to compute a rational expression
for~$\Lambda_{[D]}(z)$ for each~$[D]\in\Jac(G)$ (Theorem~\ref{thm:
  primary-secondary} and Corollary~\ref{cor: ps-genfun}).
  Proposition~\ref{prop: ps-computation} provides an effective method for
  computing primary and secondary divisors, and hence for
  computing~$\Lambda_{[D]}(z)$.  The section ends with several examples. 
\smallskip

\noindent$\bullet$ Section~\ref{sect: poly} reinterprets the results of
Section~\ref{sect: primsec} in terms of lattice points in a rational polyhedra
cone.  Generators for the cone correspond to primary divisors and lattice points
in a fundamental parallelepiped correspond to secondary divisors;  the rational
expression for~$\Lambda_{[D]}(z)$ from Section~\ref{sect: primsec} is re-derived
using standard lattice-point counting techniques (Theorem~\ref{thm: poly} and
Proposition~\ref{prop: facet}).
\smallskip

\noindent$\bullet$ Section~\ref{sect: molien} approaches our question using
invariant theory.  By Theorem~\ref{thm: invariants}, the elements of~$\E_{[D]}$
may be regarded as a basis for the (relative) polynomial invariants of a certain
complex representation of the dual of~$\Jac(G)$.  Molien's theorem then
expresses~$\Lambda_{[D]}(z)$ in a form that is substantially different from that
given earlier (Corollary~\ref{cor: lambda-molien}).  Examples are
given at the end of the section.
\smallskip

\noindent$\bullet$ Section~\ref{sect: necklaces} applies our theory to the
specific case of the cycle graph~$C_n$ with~$n$ vertices, yielding a remarkable
connection to binary necklaces.  Let~$\neck(n,k)$ denote the set of binary
necklaces with~$n$ black beads and~$k$ white beads.
Theorem~\ref{thm: permutation invariants} sets up the relevant invariant theory,
and Corollary~\ref{cor: permutation invariants} shows that~$\lambda_{[D]}(k)$
counts the number of elements of~$\neck(n,k)$
exhibiting certain symmetry (depending on~$[D]$).  In
particular,~$\lambda_{[0]}(k)$ is the total number of binary necklaces
with~$n$ black beads and~$k$ white beads.  Theorem~\ref{thm: necklace bijection}
gives a combinatorial bijection between the divisors of degree~$k$
in~$\E_{[D]}$ and~$\neck(n,k)$ for each~$[D]$ whenever~$n$ and~$k$ are
relatively prime. For further work, motivated by these results, see
\cite{Oh}. 
\smallskip

\noindent$\bullet$ Section~\ref{sect: m-matrices} generalizes the work of the
previous sections to~$M$-matrices.   These matrices, defined by certain
positivity conditions, allow one to extend much of the divisor theory of
graphs to a broader context (\cite{Gabrielov}, \cite{Guzman}).  In
Section~\ref{sect: m-matrices} we show how each~$M$-matrix gives rise to a
family of matrices---each serving the role of the Laplacian matrix and allowing our
results to be extended to this broader context.  As examples, we discuss two
particular cases from \cite{BKR}: Cartan matrices for crystallographic root
systems and McKay-Cartan matrices for faithful complex representations of
arbitrary finite groups.  
\smallskip

\noindent$\bullet$ Section~\ref{sect: further} suggests directions for further
work.


\subsection*{Acknowledgments} This work
was partially supported by a Reed College Science Research Fellowship and by the
Reed College Summer Scholarship Fund. The first author is supported by the NSF
Graduate Research Fellowship Program under Grant No.~00074041.  We thank Gopal
Goel, Gregg Musiker, and Vic Reiner for helpful discussions.  We thank Scott
Corry and an anonymous referee for their
comments.  We would also like to acknowledge our extensive use of the
mathematical software SageMath~(\cite{Sage}) and the {\em On-line Encyclopedia
of Integer Sequences} (\cite{oeis}).

\section{Divisor theory preliminaries}\label{sect: prelims}  

Let~$G=(V,\mathcal{E})$ be a connected, undirected multigraph with finite vertex
set~$V$ and finite edge multiset~$\mathcal{E}$.   Many of our constructions will
depend on fixing a vertex~$q\in V$, which we do now, once and for all.  Loops
are allowed but our results are not affected if they are removed.  We
let~$\N:=\Z_{\geq0}$ denote the natural numbers.

We recall some of the theory of divisors on graphs, referring readers unfamiliar
with this theory to~\cite{Baker} or to the textbooks~\cite{Corry}
and~\cite{Klivans}. 
A~{\em divisor} on~$G$ is an element of the free abelian group on the vertices
of~$G$,
\[
  \Div(G):=\Z V=\textstyle\left\{ \sum_{v\in V}D(v)v:D(v)\in\Z \right\}.
\]
The~{degree} of a divisor~$D$ is the sum of its
coefficients:~$\deg(D):=\sum_{v\in V}D(v)$.  For instance, if we consider $v\in
V$ as a divisor, then~$\deg(v)=1$.  We use the notation~$\deg_G(v)$ to refer to
the ordinary degree of a vertex---the number of edges incident on~$v$. The set
of divisors of degree~$k$ is denoted by~$\Div^k(G)$.  

The {\em (discrete) Laplacian operator} of~$G$ is the
function~$L\colon\Z^V\to\Z^V$ given by
\[
  L(f)(v)=\sum_{vw\in \mathcal{E}}\left( f(v)-f(w) \right)
\]
for each $f\in\Z^V$ and~$v\in V$.  The {\em divisor of a function}~$f\colon
V\to\Z$, arising by analogy from the theory of divisors on Riemann surfaces,
is then
\[
  \smalldiv(f):=\sum_{v\in V}\left( L(f)(v) \right)v\in\Div(G).
\]
The mapping~$v\mapsto\chi_v$ which sends each
vertex to its corresponding characteristic function determines an
isomorphism~$\chi\colon\Div(G)\simeq\Z^V$, and we have~$\chi\circ \smalldiv=L$,
which we use to identify~$\smalldiv$ with~$L$.

Divisors of functions are called {\em principal divisors}, and they form an
additive subgroup of~$\Div(G)$ denoted~$\Prin(G)$.  Two divisors~$D$ and~$D'$ are {\em linearly
equivalent} if~$D-D'\in\Prin(G)$, in which
case, we write~$D\sim D'$. The {\em
Picard group} of~$G$ is then the group of divisors modulo linear equivalence:
\[
  \Pic(G):=\Div(G)/\Prin(G).
\]
Since principal divisors have degree zero,~$\Pic(G)$ is graded by degree. Its
degree~$k$ part is denoted~$\Pic^k(G)$.  The degree-zero part of the Picard group
is a subgroup called the {\em Jacobian group} of~$G$:
\[
  \Jac(G):=\Pic^0(G)=\Div^0(G)/\Prin(G)\subseteq\Pic(G).
\]
We write~$[D]$ for the class of a divisor~$D$ modulo~$\Prin(G)$.  With respect
to our fixed vertex~$q$, there is an isomorphism
\begin{align}\label{eqn: Pic(G)}
  \Pic(G)&\to\Z\oplus\Jac(G)\\
  [D]&\mapsto\left(\deg(D), [D-\deg(D)q] \right).\nonumber
\end{align}

Fixing an ordering~$v_1,\dots,v_n$ of~$V$ determines a
basis for~$\Div(G)$ and a corresponding dual basis for~$\Z^V$, allowing us to
identify both spaces with~$\Z^n$.  Thus,~$D\in\Div(G)$ is identified
with~$(D(v_1),\dots,D(v_n))$, and for any~$v\in V$, we may refer to the~$v$-th
coordinate of a vector in~$\Z^n$.  With respect to the chosen bases,~$\smalldiv$
and~$L$ are represented by the~$n\times n$ {\em Laplacian matrix}, which we also
denote by~$L$.  This matrix is given by
\[
  L=\mathrm{Deg}(G)-A
\]
where $\mathrm{Deg}(G)=\diag(\deg_G(v_1),\dots,\deg_G(v_n))$
and~$A$ is the adjacency matrix for~$G$ with~$i,j$-th entry
equal to the number of edges connecting~$v_i$ to~$v_j$. The matrix~$L$ is symmetric
since~$G$ is undirected.  We then have the
isomorphism
\begin{align*}
  \Pic(G)&\simeq \cok(L)=\Z^n/\im_{\Z}(L)\\
  \textstyle [\sum_{i=1}^{n}a_iv_i]&\mapsto(a_1,\dots,a_n)+\im_{\Z}(L).
\end{align*}
The {\em reduced Laplacian} matrix for~$G$ with
respect to~$q$ is the~$(n-1)\times(n-1)$ matrix~$\tL$ formed by removing the row
and column corresponding to~$q$ from~$L$.  There is an isomorphism
\begin{align}\label{eqn: Jac(G)}
  \Jac(G)&\simeq\Z^{n-1}/\im_{\Z}(\tL)\\
  [D]&\to D|_{q=0}\nonumber
\end{align}
where~$D|_{q=0}:=\sum_{v\in V\setminus\left\{ q \right\}}D(v)v$.  The inverse
sends the class of the~$v$-th standard basis vector
in~$\Z^{n-1}/\im_{\Z}(\tL)$ to~$[v-q]$ for each~$v\neq q$.
Isomorphisms~\ref{eqn: Pic(G)} and~\ref{eqn: Jac(G)} combine to say
that for~$D,D'\in\Div(G)$,
\[
  D\sim D'
  \quad\Longleftrightarrow\quad 
  \left(
  \deg(D)=\deg(D')
  \quad\text{and}\quad
  D|_{q=0}=D'|_{q=0}\bmod\im_{\Z}(\tL)\right).
\]

The kernel of the Laplacian matrix is the set of constant vectors, and the
reduced Laplacian has full rank~$n-1$.  By the matrix-tree
theorem, the number of spanning trees of~$G$ is~$\det(\tL)$, and thus
by~$\eqref{eqn: Jac(G)}$, it is also the order of~$\Jac(G)$.  We adopt the
following notation: for each~$v\in V$, let
\[
\ord_q(v):= \text{order of~$[v-q]\in\Jac(G)$}.
\]
In particular,~$\ord_q(q)=1$.

We now describe a standard set of representatives for the elements of~$\Jac(G)$.
A {\em set firing} by a subset $W\subseteq V$ on a divisor~$D$ produces a new
divisor~$D'=D-\smalldiv(\chi_W)$ where~$\chi_W$ is the characteristic function
of~$W$.  Having fixed an ordering of the vertices, we identify~$W$ with
a~$0$\,-$1$ vector in~$\Z^n$, and we have~$D'=D-LW$ where~$L$ is the Laplacian matrix.  
A {\em reverse firing} would instead produce the divisor~$D+LW$.  Thus, two
divisors are linearly equivalent if and only if they differ by a sequence of set
firings and reverse firings.

Firing a set~$W$ is {\em legal} if~$D'(w)\geq0$ for all~$w\in W$.
The divisor~$D$ is~{\em $q$-reduced} if
\begin{enumerate}
  \item[(i)] $D(v)\geq 0$ for all~$v\in V\setminus\left\{ q \right\}$, and 
  \item[(ii)]  $D$ has no legal set firing by a nonempty set $W\subseteq V\setminus\left\{ q \right\}$.
\end{enumerate}
It turns out that each divisor is
linearly equivalent to a unique~$q$-reduced divisor. Thus, the~$q$-reduced
divisors of degree~$0$ form a set of representatives for the elements
of~$\Jac(G)$.  There is an efficient algorithm (Dhar's algorithm)
for finding the~$q$-reduced representative of any divisor class.  If~$D$ is
$q$-reduced, then
letting~$\deg(D|_{q=0}):=\sum_{v\in V\setminus\left\{ q \right\}}D(v)$ we have
\[
  0\leq\deg(D|_{q=0})\leq |\mathcal{E}|-|V|+1.
\]
Therefore, searching through all divisors~$D$ of degree~$0$ satisfying the above
bound provides a fairly efficient means of calculating~$\Jac(G)$.  (For an
improvement, see~\cite{Shokrieh}.)

\subsection{Partitioning effective divisors}
A divisor~$E$ is {\em effective} if~$E(v)\geq0$ for all~$v\in V$, in which case
we write~$E\geq0$.  The {\em complete linear system} of a divisor~$D$ is its set
of linearly equivalent effective divisors:
\[
  |D| := \left\{ E\in\Div(G): E\geq 0\text{ and } E\sim D \right\}.
\]
Note that~$|D|$ depends only on the divisor class of~$D$.  Also, since linearly
equivalent divisors have the same degree,~$|D|$ is finite.

For each~$[D]\in\Jac(G)$, define
\[
  \E_{[D]}:=\cup_{k\geq0}|D+kq|=\left\{ E\in\Div(G):E\geq0\text{ and }
  E-\deg(E)q\sim D \right\}.
\]
The~$\E_{[D]}$ partition the set of effective divisors as~$D$ runs over a set
of representatives for~$\Jac(G)$.  The collection~$\E_{[0]}$ is a commutative monoid, and
it acts on each~$\E_{[D]}$ via addition: $\E_{[0]}+\E_{[D]}=\E_{[D]}$.
Note that~$\E_{[D]}$ depends on~$q$.\footnote{For~$q'\in V$, writing~$D+kq =
  D+kq'+k(q-q')$ shows the dependence is ``periodic'' with period equal to the
  order of~$[q-q']\in\Jac(G)$.}

\begin{defn}
  The~$\lambda$-sequence for~$[D]\in\Jac(G)$ is the
  sequence with~$k$-th term
  \[
    \lambda_{[D]}(k):=\#|D+kq|.
  \]
  (It does not depend on the choice of representative of the
  class~$[D]$.)  The generating function for the~$\lambda$-sequence is
  \[
    \Lambda_{[D]}(z):=\sum_{k\geq0}\lambda_{[D]}(k)z^k.
  \]
\end{defn}
Our main goal is to find closed expressions for~$\Lambda_{[D]}$ for
each~$[D]\in\Jac(G)$ and thus determine the cardinality of~$|F|$ for
all~$F\in\Div(G)$.

\section{Primary and secondary divisors}\label{sect: primsec}
In this section, we compute $\Lambda_{[D]}$ using \emph{primary} and
\emph{secondary} divisors, defined as part of the following theorem. 

\Needspace{3\baselineskip}
\begin{thm}\label{thm: primary-secondary}\ 
\begin{enumerate}[label=\rm{(\arabic*)},leftmargin=*]
  \item\label{ps-1} {\rm (Existence)} There exists a finite subset
    $\primary\subset\E_{[0]}$ and for each $[D]\in\Jac(G)$, a finite
    subset $\secondary_{[D]}\subset\E_{[D]}$ such
    that each $E\in\E_{[D]}$ can be written uniquely as
\[
E = F + \sum_{P\in\primary}a_P P
\]
with $F\in\secondary_{[D]}$ and~$a_P\in\N$ for all~$P\in\primary$. The set
$\primary$ is called a set of {\em primary divisors} for~$G$,  and
$\secondary_{[D]}$ is called the set of {\em $[D]$-secondary} divisors with
respect to~$\primary$.
\item\label{ps-2} {\rm (Uniqueness)} Sets~$\primary$
  and~$\{\secondary_{[D]}\}_{[D]\in\Jac(G)}$ satisfy part~\ref{ps-1} if and only
  if
  \[
    \primary=\left\{\tl_v v: v\in V\right\}
    \quad\text{and}\quad
    \secondary_{[D]}=\{E\in\E_{[D]}:E(v)<\tl_v \mbox{ for all $v\in
    V$}\},
  \]
  where~$\tl_v$ is a positive multiple of~$\ord_{q}(v)$ for all~$v\in
  V$. In particular, taking~$\ell_v=\ord_q(v)$ for all~$v\in V$ produces the set
  of primary divisors of smallest degree and corresponding sets of secondary
  divisors with minimal cardinality.
\end{enumerate}
\end{thm}
\begin{proof} To prove part~\ref{ps-1}, for each~$v\in V$, let~$\tl_v$ be a
  positive multiple of~$\ord_{q}(v)$, and define~$\primary$ and
  each~$\secondary_{[D]}$ as in part~\ref{ps-2} of the theorem.
  Given~$E\in\E_{[D]}$, for each~$v\in V$, let~$k_v$ be the largest integer such
  that~$E(v)-k_v\tl_v\geq0$, and define~$F:=E-\sum_{v\in
  V}k_v\tl_vv\in\secondary_{[D]}$.  Then $E=F+\sum_{v\in V}k_v\tl_vv$ is a
  decomposition as required in part~\ref{ps-1}.   For uniqueness of this
  decomposition, suppose $E = F' + \sum_{v\in V}k_v'v$ for
  some~$F'\in\secondary_{[D]}$ and~$k'_v\in\N$.  Then for each~$v\in V$, we
  have~$0\leq F(v)=E(v)- k_v\tl_v<\tl_v$ and~$0\leq F'(v)=E(v)-k'_v\tl_v<\tl_v$.
  Subtracting these inequalities yields $-\tl_v<(k'_v-k_v)\tl_v<\tl_v$.  It
  follows that~$k_v=k'_v$ for all~$v$ and~$F=F'$.

  We have shown that if~$\primary$ and~$\secondary_{[D]}$ have the form
  displayed in~$\ref{ps-2}$, then they serve as sets of primary and secondary
  divisors, i.e., they satisfy the conditions in part~\ref{ps-1}.  To show that
  necessity of this form and thus finish the proof of part~\ref{ps-2},
  let~$\primary$ and~$\{\secondary_{[D]}\}_{[D]\in\Jac(G)}$ be any sets of
  primary and~$[D]$-secondary divisors. Since~$\secondary_{[0]}$ is finite, for
  each~$v\in V$, there is a smallest natural number~$m_v$ such
  that~$m_v\ord(v)v\notin\secondary_{[0]}$.  Consider the primary-secondary
  decomposition~$m_v\ord(v)v=F+\sum_{P\in\primary}a_PP$. Since the divisors on
  the right-hand side are effective and all~$a_P$ are nonnegative, considering
  coefficients on both sides, it follows that decomposition takes the
  form~$m_v\ord(v)v=av+bv$ for some~$a,b\in\N$ such that $av\in\secondary_{[0]}$
  and $bv\in\primary$. By definition of~$m_v$, we must have~$b>0$.
  Since~$\secondary_{[0]}$ and~$\primary$ are subsets of~$\E_{[0]}$, we
  have~$av\sim aq$ and~$bv\sim bq$.  Therefore,~$a=a'\ord(v)$ and~$b=b'\ord(v)$
  for some~$a',b'\in\N$.  If~$a'\neq0$, then~$b'<m_v$, which
  implies~$b'\ord(v)v\in\secondary_{[0]}$ by definition of~$m_v$.  However, that
  is impossible since the uniqueness of decompositions described in
  part~\ref{ps-1} implies~$\secondary_{[0]}$ and~$\primary$ are disjoint.
  Defining~$\tl_v:=m_v\ord(v)$, it follows that~$\tl_v v\in\primary$ for
  all~$v$.  However, again by uniqueness of decompositions, the elements
  of~$\primary$ must be linearly independent over~$\Z$, which implies there are
  no other primary divisors.  So~$\primary=\left\{ \tl_v:v\in V \right\}$, as
  claimed, and it is then straightforward to show that~$\secondary_{[D]}$ must
  have the form stated in~\ref{ps-2} for each~$[D]\in\Jac(G)$.
\end{proof}

\begin{cor}\label{cor: ps-genfun} Fix primary and secondary divisors as in Theorem~\ref{thm:
  primary-secondary}. For each~$[D]\in\Jac(G)$,
  \[
    \Lambda_{[D]}(z)=\frac{S(z)}{\prod_{v\in V}\left(1-z^{\tl_v}\right)}
  \]
  where
  \[
    S(z)=\sum_{F\in\secondary_{[D]}}z^{\deg(F)}.
  \]
\end{cor}
\begin{proof} Introduce indeterminates~$\left\{ x_v \right\}_{v\in V}$, and
  identify each effective divisor~$E$ with a monomial~$x^E:=\prod_{v\in
  V}x_v^{E(v)}$. Define
  \[
    \sigma_D(x)=\sum_{E\in\E_{[D]}}x^{E}.
  \]
  By Theorem~\ref{thm: primary-secondary}, we may uniquely write
  \[
    x^E=x^F\cdot\prod_{P\in\primary}x^{a_PP} =x^F\cdot\prod_{v\in V}x_v^{a_v\tl_v}
  \]
  for some~$F\in\secondary_{[D]}$ and~$a_v\geq0$.
  Then
\begin{align*}
\sigma_D(x)&=\sum_{F\in\secondary_{[D]}}x^F\sum_{a\in \N^V}\prod_{v\in
V}x_v^{a_v\tl_v}
=\left(\sum_{F\in\secondary_{[D]}}x^F\right)\prod_{v\in
V}\left(1+x_v^{\tl_v}+x_v^{2\tl_v}
+\dots\right)\\[8pt]
&=\left(\sum_{F\in\secondary_{[D]}}x^F\right)\prod_{v\in
V}\frac{1}{1-x_v^{\tl_v}}.
\end{align*}
Now note that $\Lambda_{[D]}(z)=\sigma_D(z,z,\dots,z)$ to conclude the proof.
\end{proof}

\begin{remark}\label{remark: ps-comp}  Denote the numerator~$S$ in Corollary~\ref{cor: ps-genfun}
  by~$S_{[D]}(z)$ to indicate its dependence on $[D]\in\Jac(G)$.  Since
  the~$\E_{[D]}$ partition the set of effective divisors, it follows
  that~$\sum_{[D]\in\Jac(G)}\Lambda_{[D]}(z)=1/(1-z)^n$, and hence,
  \[
    \sum_{[D]\in\Jac(G)}S_{[D]}(z)=\frac{\prod_{v\in V}(1-z^{\tl_v})}{(1-z)^n}=\prod_{v\in
      V}(1+z+z^2+\dots+z^{\tl_v-1}).
  \]
\end{remark}

We now describe how to easily compute primary and secondary divisors. 
Recall that we have fixed an ordering of the vertices of~$G$ to identify~$\Div(G)$
with~$\Z^n$.  Fix primary and secondary divisors as in
Theorem~\ref{thm: primary-secondary}~\ref{ps-2}, and consider the natural projection
\[
  \pi\colon \Div(G)=\Z^n\to\Z^n/\prod_{v\in V}\tl_v\Z.
\]
A {\em standard representative} of an element~$\overline{D}\in\Z^n/\prod_{v\in
V}\tl_v\Z$ is a divisor~$E\in\Div(G)$ such that~$\pi(E)=\overline{D}$ and~$0\leq
E(v)<\tl_v$ for all~$v$.  

Assume vertex~$q$ appears last in the ordering so that
\[
  \im_{\Z}\tL\times\Z\subseteq\Z^{n-1}\times\Z=\Z^n=\Div(G).
\]
For each~$D\in\Div(G)$, define
\[
  H_{[D]}:=\pi(D+(\im_{\Z}\tL\times\Z))\subseteq\Z^n/\prod_{v\in V}\tl_v\Z.
\]
\begin{prop}\label{prop: ps-computation} Let~$\primary=\left\{ \tl_v v:v\in V
  \right\}$ and~$\secondary_{[D]}$ for each~$[D]\in\Jac(G)$ be as in
  Theorem~\ref{thm: primary-secondary}~\ref{ps-2}.
  \begin{enumerate}[label=\rm{(\arabic*)},leftmargin=*]
    \item\label{ps-comp1}  Let~$\tL^{-1}$ be the inverse of the reduced Laplacian
      over~$\Q$. Then, for each~$v\neq q$, the integer~$\ord_{q}(v)$ is the least common multiple
      of the denominators of the (reduced) fractions in the~$v$-th column
      of~$\tL^{-1}$.
    \item\label{ps-comp2} For each~$[D]\in\Jac(G)$, there is a bijection of sets
      \begin{align*}
	\secondary_{[D]}&\to H_{[D]}\\
	E&\mapsto\pi(E|_{q=0},\deg(E)),
      \end{align*}
      and thus~$\secondary_{[D]}$ is exactly a set of standard representatives
      for~$H_{[D]}$.
    \item\label{ps-comp3} For each~$[D]\in\Jac(G)$,
      \[
	|\secondary_{[D]}||\Jac(G)|=\prod_{v\in V}\tl_v.
      \]
  \end{enumerate}
\end{prop}
\begin{proof} First note that~$\tL$ has rank~$n-1$, and thus has an
  inverse~$\tL^{-1}$ over~$\Q$.  By~\eqref{eqn: Jac(G)} of Section~\ref{sect:
  prelims}, the order of~$[v-q]\in\Jac(G)$ is the least positive integer~$k$
  such that~$kv\in\im_{\Z}\tL$.  Therefore,~$\ord_{q}(v)$ is the least positive
  integer~$k$ such that~$\tL^{-1}(kv)\in\Z^{n-1}$. Part~\ref{ps-comp1} follows.

  Part~\ref{ps-comp2} is immediate: a divisor~$E$ is a standard representative
  for an element in~$H_{[D]}$ if and only if $E|_{q=0}=D|_{q=0} \bmod\im_{\Z}\tL$ and
  $0\leq E(v)<\tl_v$ for all~$v\in V$, which is exactly the requirement for
  being an element of~$\secondary_{[D]}$.

  Now consider part~\ref{ps-comp3}.
  Since~$\tl_v[v-q]=[0]\in\Jac(G)$ for all~$v$, there is a surjection
  \[
    \textstyle \Z^n/\prod_{v\in V}\tl_v\Z\to\Jac(G)\simeq\Z^{n-1}/\im_{\Z}\tL
  \]
  which sends the class of the~$v$-th standard basis vector to~$[v-q]$.
  Its kernel is~$H_{[0]}$.  So by part~\ref{ps-comp2}, we have
  $|\secondary_{[0]}||\Jac(G)|=\prod_{v\in V}\tl_v$.  However, for
  each~$[D]\in\Jac(G)$, there is a
  well-defined bijection
  \begin{align*}
    H_{[0]}&\to H_{[D]}\\
    \pi(E)&\mapsto\pi(D+E).
  \end{align*}
  So~$|\secondary_{[D]}|=|\secondary_{[0]}|$, and part~\ref{ps-comp3} follows.
\end{proof}
\begin{remark}\label{remark: method} (Computation of primary and secondary
  divisors) To summarize the above: in order to compute a set of primary
  divisors, use Proposition~\ref{prop: ps-computation} to compute
  each~$\ord_{q}(v)$ for~$v\neq q$ from the columns of~$\tL^{-1}$. Then
  take~$\primary=\{ \tl_v v \}_{v\in V}$ where the~$\tl_v$ are arbitrary
  positive multiples of the corresponding~$\ord_{q}(v)$.  In order to minimize
  the number of secondary divisors, one would take~$\tl_v=\ord_{q}(v)$ for
  each~$v$.  In particular, this would mean~$\tl_q=\ord_{q}(q)=1$.

  Next, use part~\ref{ps-comp2} of Proposition~\ref{prop: ps-computation} to
  compute~$\secondary_{[D]}$ for each~$[D]\in\Jac(G)$. To ease the computation
  of~$\im_{\Z}\tL$, perform invertible integer column operations on~$\tL$ to
  compute its Hermite normal form~$A$. (We will always take ``Hermite normal
  form'' to mean ``column Hermite normal form''.) Then find the set~$S$ of standard
  representatives for the coset $D|_{q=0}+\im_{\Z}A$ modulo $\prod_{v\in
  V\setminus\left\{ q \right\}}\tl_v\Z$.  Finally~$\secondary_{[D]}=\left\{
  c+kq:c\in S \text{ and } 0\leq k <\tl_q\right\}$.
  \end{remark}

\Needspace{3\baselineskip}
  \subsection{Examples}  We now use the method outlined in Remark~\ref{remark:
  method} to compute~$\lambda$-sequence generating functions for several
  examples.
\Needspace{3\baselineskip}
\subsubsection{Trees}\label{example: ps-tree} If~$G$ is a tree, then~$\Jac(G)$
is trivial, and the mapping~$[D]\mapsto\deg(D)$ is an isomorphism of~$\Pic(G)$
with~$\Z$.  It follows that for any~$q\in V$,
\[
  \E_{[0]} = \cup_{k\geq 0}|kq| = \left\{ E\in\Div(G):E\geq0 \text{ and }
  \deg(E)\geq0 \right\}.
\]
So, letting~$n=|V|$, the cardinality of~$|kq|$ is the number of elements of~$\N^n$
with coordinate sum equal to~$k$. Thus,
\[
  \Lambda_{[0]}(z)=\sum_{k\geq 0}\binom{n-1+k}{k}z^k=\frac{1}{(1-z)^n},
\]
in agreement with Corollary~\ref{cor: ps-genfun}
where we take~$\tl_v=1$ for all~$v\in V$.  In that case~$\primary=V$
and~$\secondary_{[0]}=\left\{ 0 \right\}$.  

\subsubsection{Diamond graph.}\label{example: ps-diamond}  Let~$G$ be the diamond graph pictured in
Figure~\ref{fig: diamond graph}.
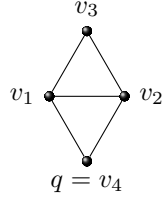
\begin{figure}[ht]
\centering
\begin{tikzpicture}[scale=0.5]
\node[ball,label=$v_3$] (v3) at (0,1.732) {};
\node[ball,label=180:$v_1$] (v1) at (-1,0) {};
\node[ball,label=0:$v_2$] (v2) at (1,0) {};
\node[ball,label=270:{$q=v_4$}] (v4) at (0,-1.732) {};
\draw (v3)--(v1)--(v4)--(v2)--(v3);
\draw (v1)--(v2);
\end{tikzpicture}
\caption{The diamond graph.}\label{fig: diamond graph}
\end{figure}
The reduced Laplacian for $G$ and its inverse are:
\[
\tL=
\left(\begin{array}{rrr}
3 & -1 & -1 \\
-1 & 3 & -1 \\
-1 & -1 & 2
\end{array}\right),
\quad
\tL^{-1}=
\left(\begin{array}{rrr}
\frac{5}{8} & \frac{3}{8} & \frac{1}{2} \\[5pt]
\frac{3}{8} & \frac{5}{8} & \frac{1}{2} \\[5pt]
\frac{1}{2} & \frac{1}{2} & 1
\end{array}\right).
\]
Taking the least common multiples of denominators in the columns of $\tL^{-1}$ gives 
\[
  (\ord_{q}(v_1),\ord_{q}(v_2),\ord_{q}(v_3),\ord_{q}(q)) =(8,8,2,1).
\]
To minimize the number of secondary divisors, we take~$\tl_v=\ord_{q}(v)$
for all~$v$.  Thus,
\[
  \primary=\left\{8v_1,8v_2,2v_3,q \right\}.
\]
By~Proposition~\ref{prop: ps-computation}~\ref{ps-comp3}, we
have~$|\secondary_{[D]}|=16$ for each~$[D]\in\Jac(G)$ since
$|\Jac(G)|=\det(\tL)=8$.

To compute~$\secondary_{[0]}$, perform invertible integer column operations on $\tL$ to reduce it to Hermite normal form:
\[
  A =
\left(\begin{array}{rrr}
1 & 0 & 0 \\
1 & 4 & 0 \\
0 & 1 & 2
\end{array}\right).
\]
Using this matrix, it is easy to find standard representatives for~$\im_{\Z}A=\im_{\Z}\tL$ 
modulo~$8\Z\times8\Z\times2\Z$.  According to Remark~\ref{remark: ps-comp}, since~$\tl_q=1$,
we then append~$0$ to each of these representatives to get
\begin{align*}
  \secondary_{[0]}=\{ &(0,0,0,0),(1,1,0,0),(2,2,0,0),(3,3,0,0),(4,4,0,0),(5,5,0,0),(6,6,0,0),(7,7,0,0)\\[8pt]
  &(0,4,1,0),(1,5,1,0),(2,6,1,0),(3,7,1,0),(4,0,1,0),(5,1,1,0),(6,2,1,0),(7,3,1,0)\}.
\end{align*}
From Corollary~\ref{cor: ps-genfun},
\begin{align*}
\Lambda_{[0]}(z)&=\frac{1+z^2+z^4+2z^5+z^6+2z^7+z^8+2z^9+z^{10}+2z^{11}+z^{12}+z^{14}}{(1-z)(1-z^2)(1-z^8)^2}\\[5pt]
&=\frac{1-z+z^2-z^3+z^4+z^5-z^6+z^7}{(1+z)^2(1+z^2)(1+z^4)(1-z)^4}\\[8pt]
&=1 + z + 3 z^{2} + 3 z^{3} + 6 z^{4} + 8 z^{5} + 12 z^{6} + 16 z^{7} + 23 z^{8}
+ 29 z^{9} + 39 z^{10} +\cdots.
\end{align*}
For instance, the six effective divisors of degree~$4$ in~$\E_{[0]}$ predicted by
the generating function are
\[
  (0,0,4,0),(0,0,2,2),(0,0,0,4),(1,1,2,0),(1,1,0,2),(2,2,0,0),
\]
which we get from~$\secondary_{[0]}$ by adding appropriate multiples of elements
of~$\primary=\{8v_1,8v_2,2v_3,q\}$.

As another example, let $D=v_1-q=(1,0,0,-1)$.  To find~$\secondary_{[D]}$, 
add~$D$ to each of the divisors in~$\secondary_{[0]}$, then take their standard
representatives as elements of~$\Z_8\times\Z_8\times\Z_2\times\Z_1$:
\begin{align*}
  \secondary_{[D]}= \{ &(1,0,0,0),(2,1,0,0),(3,2,0,0),(4,3,0,0),
    (5,4,0,0),(6,5,0,0),(7,6,0,0),(0,7,0,0)\\[8pt]
&(1,4,1,0),(2,5,1,0),(3,6,1,0),(4,7,1,0),
(5,0,1,0),(6,1,1,0),(7,2,1,0),(0,3,1,0)\}.
\end{align*}
Therefore,
\begin{align*}
\Lambda_{[D]}(z)&=\frac{z+z^3+z^4+z^5+2z^6+2z^7+2z^8+z^9+2z^{10}+z^{11}+z^{12}+z^{13}}{(1-z)(1-z^2)(1-z^8)^2}\\[5pt]
&=\frac{x(1-x+x^3)}{(1+x)^2(1+x^4)(1-x)^4}\\[8pt]
&= z + z^{2} + 3 z^{3} + 4 z^{4} + 7 z^{5} + 10 z^{6} + 15 z^{7} + 20 z^{8} + 28
z^{9} + 35 z^{10}+\cdots.
\end{align*}

\subsubsection{Cycle graphs.}\label{example: ps-cycle} Let~$C_n$ be the cycle graph on~$n$ vertices, with
vertices~$v_1,\dots,v_n$ around the cycle.  Take~$q=v_n$. It is well-known
that~$\Jac(C_n)\simeq\Z/n\Z$, with generator~$D_1:=[v_1-q]$ and such
that $D_j:=[v_{j}-q]=j[v_1-q]$ for all~$j$, where the indices are determined
modulo~$n$.  Therefore,~$\ord_{q}(v_i)=n/\gcd(i,n)$ for all~$i$.  For
convenience, take~$\tl_{v_i}=n$ for~$i=1,\dots,n-1$ and~$\tl_q=1$.  The reduced
Laplacian~$\tL$ is the~$(n-1)\times(n-1)$ tridiagonal matrix with~$2$s on the
diagonal and~$-1$s on the super and subdiagonals.  It is straightforward to
reduce~$\tL$ to its Hermite form, which is~$I_{n-1}+B_{n-1}$ where~$I_{n-1}$ is the
identity matrix and~$B_{n-1}$ is a matrix whose rows are all~$0$-vectors except for
the last row, which is the vector~$(1,2,\dots,n-1)$.  See Figure~\ref{fig:
ps-cycle} for an example.
\begin{figure}[ht]
\centering
\begin{tikzpicture}[scale=0.7]
  \draw (0,0) circle (1);
  \foreach \x in {0,1,...,4} {
    \node[ball] (\x) at (72*\x-90:1) {};
  }
  \foreach \x in {1,...,4} {
    \node at (72*\x-90:1.45) {$v_{\x}$};
  }
  \node at (270:1.45) {$q$};
  \node at (6,0) {$
  \left(\begin{array}{rrrr}
    2 & -1 & 0 & 0\\
    -1 & 2 & -1 & 0\\
    0 & -1 & 2 & -1\\
    0 & 0 & -1 & 2\\
    \end{array}\right)
    $};
  \node at (13,0) {$
  \left(\begin{array}{rrrrr}
  1 & 0 & 0 & 0\\
  0 & 1 & 0 & 0\\
  0 & 0 & 1 & 0\\
  1 & 2 & 3 & 5
  \end{array}\right)
  $};
  \node at (0,-2.3) {$C_5$};
  \node at (6,-2) {$\tL$};
  \node at (13,-2) {Hermite form for~$\tL$};
\end{tikzpicture}
\caption{The cycle graph~$C_5$ (cf.~Example~\ref{example: ps-cycle}).}\label{fig: ps-cycle}
\end{figure}
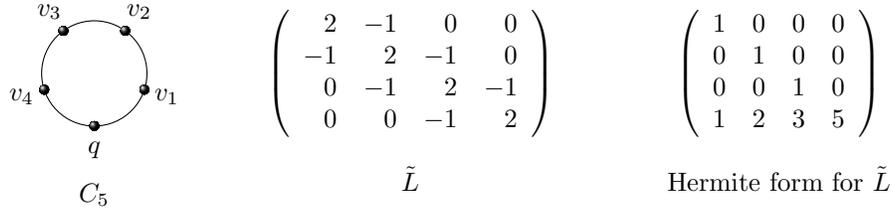
The primary divisors are~$\primary=\{nv_1,nv_2,\dots,nv_{n-1},q\}$ and for
each~$j=0,1,\dots,n-1$, the
secondary divisors for~$D_j$ are
\[
  \secondary_{[D_j]}=
  \textstyle
  \left\{ \left((a_1+j)\bmod n,a_2,\dots,a_{n-2},\sum_{i=1}^{n-2}ia_i\bmod n,0\right):0\leq
  a_i<n\text{ for~$i=1,\dots,n-2$} \right\}.
\]
For example, in the case~$n=4$ and~$j=2$, we have~$16$ secondary divisors
for~$[D_2]\in\Jac(C_4)$:
\begin{align*}
  \secondary_{[D_2]}= \{ &
    (2, 0, 0, 0), (2, 1, 2, 0), (2, 2, 0, 0), (2, 3, 2, 0), (3, 0, 1, 0), (3, 1, 3, 0),
    (3, 2, 1, 0), (3, 3, 3, 0),
    \\[8pt]
    & (0, 0, 2, 0), (0, 1, 0, 0), (0, 2, 2, 0), (0, 3, 0, 0), (1, 0, 3, 0), (1, 1, 1, 0),
    (1, 2, 3, 0),
  (1, 3, 1, 0) \}.
\end{align*}
By Corollary~\ref{cor: ps-genfun},
\begin{align*}
  \Lambda_{[D_2]}(z)&=\frac{z+2z^2+2z^3+4z^4+2z^5+2z^6+2z^7+z^9}{(1-z^3)^4(1-z)}\\[8pt]
  &=\frac{z+z^2-z^3+z^4}{(1+z^2)(1+z)^2(1-z)^4}\\[8pt]
  &= z+3z^2+5z^3+9z^4+14z^5+22z^6+30z^7+42z^8+55z^9+73z^{10}+\cdots.
\end{align*}
For instance, the term $5z^3$ in the above expression corresponds to the~$5$
elements of the complete linear system for~$D_2+3q=(2,0,0,1)$ pictured in
Figure~\ref{fig: D2+3q}.
\begin{figure}[ht]
\centering
\begin{center}
  \begin{tikzpicture}[scale=0.62]
    \begin{scope}
      \node[miniball,label={below:{$1$}}] (0) at (0,-1) {};   
      \node[miniball,label={right:{$2$}}] (1) at (0.62,0) {};   
      \node[miniball,label={above:{$0$}}] (2) at (0,1) {};   
      \node[miniball,label={left:{$0$}}] (3) at (-0.62,0) {};
      \draw (0)--(1)--(2)--(3)--(0);
    \end{scope}
    \begin{scope}[xshift=4cm]
      \node[miniball,label={below:{$2$}}] (0) at (0,-1) {};   
      \node[miniball,label={right:{$0$}}] (1) at (0.62,0) {};   
      \node[miniball,label={above:{$1$}}] (2) at (0,1) {};   
      \node[miniball,label={left:{$0$}}] (3) at (-0.62,0) {};
      \draw (0)--(1)--(2)--(3)--(0);
    \end{scope}
    \begin{scope}[xshift=8cm]
      \node[miniball,label={below:{$0$}}] (0) at (0,-1) {};   
      \node[miniball,label={right:{$1$}}] (1) at (0.62,0) {};   
      \node[miniball,label={above:{$1$}}] (2) at (0,1) {};   
      \node[miniball,label={left:{$1$}}] (3) at (-0.62,0) {};
      \draw (0)--(1)--(2)--(3)--(0);
    \end{scope}
    \begin{scope}[xshift=12cm]
      \node[miniball,label={below:{$1$}}] (0) at (0,-1) {};   
      \node[miniball,label={right:{$0$}}] (1) at (0.62,0) {};   
      \node[miniball,label={above:{$0$}}] (2) at (0,1) {};   
      \node[miniball,label={left:{$2$}}] (3) at (-0.62,0) {};
      \draw (0)--(1)--(2)--(3)--(0);
    \end{scope}
    \begin{scope}[xshift=16cm]
      \node[miniball,label={below:{$0$}}] (0) at (0,-1) {};   
      \node[miniball,label={right:{$0$}}] (1) at (0.62,0) {};   
      \node[miniball,label={above:{$3$}}] (2) at (0,1) {};   
      \node[miniball,label={left:{$0$}}] (3) at (-0.62,0) {};
      \draw (0)--(1)--(2)--(3)--(0);
    \end{scope}
  \end{tikzpicture}
\end{center}
\caption[Complete linear system]{The complete linear system~$|D_2+3q|$ on~$C_4=
  \begin{tikzpicture}[scale=0.3,baseline=-3pt]
      \node[miniball,label={below:{\small $q$}}] (0) at (0,-1) {};   
      \node[miniball,label={right:{\small $v_1$}}] (1) at (0.62,0) {};   
      \node[miniball,label={above:{\small $v_2$}}] (2) at (0,1) {};   
      \node[miniball,label={left:{\small $v_3$}}] (3) at (-0.62,0) {};
      \draw (0)--(1)--(2)--(3)--(0);
    \end{tikzpicture}
  $ (cf.~Example~\ref{example: ps-cycle}).}\label{fig: D2+3q}
\end{figure}

See Section~\ref{sect: necklaces} for the relation between complete linear
systems on cycle graphs and binary necklaces.

\subsubsection{Complete graphs.}\label{example: ps-complete} Let~$K_n$ be the
complete graph with vertex set~$V=\{v_1,\dots,v_n\}$, and let~$q:=v_n$. Its reduced
Laplacian~$\tL$ is the matrix~$nI_{n-1}-J_{n-1}$ where~$J_{n-1}$ is
the~$(n-1)\times(n-1)$ matrix whose entries are all~$1$. The inverse
is~$\tL^{-1}=\frac{1}{n}(I_{n-1}+J_{n-1})$, and therefore,~$\ord_{q}(v)=n$ for all
vertices~$v\neq q$. To reduce~$\tL$ to Hermite normal form, add
columns~$2$ through~$n-1$ to the first column of~$\tL$,
then add the first column to each of the others.  The result is the matrix
formed by replacing the first column of~$nI_{n-1}$ by a column of all~$1$s. So
the image of~$\tL$ in~$(\Z/n\Z)^{n-1}$ is spanned by the vector of all~$1$s.
See Figure~\ref{fig: ps-complete} for an example.  
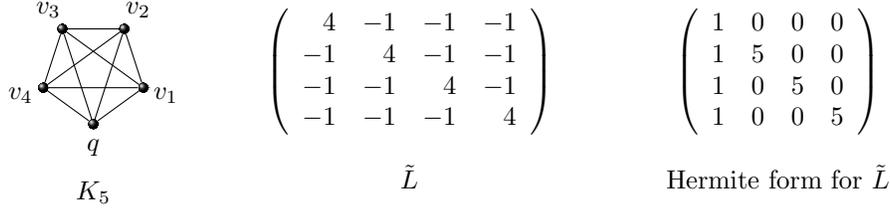
\begin{figure}[ht]
\centering
\begin{tikzpicture}[scale=0.7]
  \foreach \x in {0,...,4} {
    \node[ball] (\x) at (72*\x-90:1) {};
  }
  \foreach \x in {1,...,4} {
    \node at (72*\x-90:1.45) {$v_{\x}$};
  }
  \node at (270:1.45) {$q$};
  \foreach \x in {0,...,3}
    \foreach \y in {\x,...,4} {
      \draw (\x)--(\y);
  };
  \node at (6,0) {$
  \left(\begin{array}{rrrr}
    4 & -1 & -1 & -1\\
    -1 & 4 & -1 & -1\\
    -1 & -1 & 4 & -1\\
    -1 & -1 & -1 & 4\\
    \end{array}\right)
    $};
  \node at (13,0) {$
  \left(\begin{array}{rrrrr}
  1 & 0 & 0 & 0\\
  1 & 5 & 0 & 0\\
  1 & 0 & 5 & 0\\
  1 & 0 & 0 & 5
  \end{array}\right)
  $};
  \node at (0,-2.3) {$K_5$};
  \node at (6,-2) {$\tL$};
  \node at (13,-2) {Hermite form for~$\tL$};
\end{tikzpicture}
\caption{The complete graph~$K_5$ (cf.~Example~\ref{example:
ps-complete}).}\label{fig: ps-complete}
\end{figure}

Taking~$\tl_v=n$ for~$v\neq q$
and~$\tl_q=1$, the primary divisors are~$\primary=\{nv\}_{v\in V\setminus\left\{
q \right\}}\cup\{q\}$, and the secondary divisors for~$[0]\in\Jac(K_n)$ are
\[
  \secondary_{[0]}=\{k(1,1,\dots,1,0): 0\leq k<n\}.
\]
Hence,
\begin{align*}
  \Lambda_{[0]}(z)
  &= \frac{\sum_{k=0}^{n-1}z^{k(n-1)}}{(1-z^n)^{n-1}(1-z)}\\[8pt]
  &=\frac{1}{1-z}\left(\,\sum_{k=0}^{n-1}z^{k(n-1)}\right)\sum_{k\geq
  0}\binom{n-2+k}{n-2}z^{kn}.
\end{align*}
The sequence of first differences of the~$\lambda$-sequence for~$0$
has~$k$-th term~$\Delta\lambda_{[0]}(k):=\lambda_{[0]}(k+1)-\lambda_{[0]}(k)$,
which is the coefficient of~$z^{k+1}$ in~$(1-z)\Lambda_{[0]}(z)$ for~$k\geq0$.  
A bit of calculation then shows that writing~$k=an+b$ in terms of quotient and
remainder by~$n$ gives
\[
 \Delta\lambda_{[0]}(k)=\binom{a+b}{n-2}.
\]
This means $\Delta\lambda_{[0]}$ is formed by concatenating sequences of length~$n$:
\[
  \binom{m}{n-2},\binom{m+1}{n-2},\dots,\binom{m+(n-1)}{n-2}  
\]
for $m\in\N$.
For example, in the case~$n=5$, 
\begin{align*}
  \Lambda_{[0]}(z)&=\frac{1+z^4+z^{8}+z^{12}+z^{16}}{(1-z^5)^4(1-z)}\\[8pt]
  &=
1+z+z^{2}+z^{3}+2z^{4}+6z^{5}+6z^{6}+6z^{7}+7z^{8}+11z^{9}+\cdots.
\end{align*}
We have
\[
  \lambda_{[0]}=
  1, 1, 1, 1, 2, 6, 6, 6, 7, 11, 21, 21, 22, 26, 36, 56, 57, 61, 71, 91, 126,
  130, 140, 160, 195, 251, \dots,
\]
which gives
\[
\Delta\lambda_{[0]}=
\underline{0, 0, 0, 1, 4},\ \underline{0, 0, 1, 4, 10},\ \underline{0, 1, 4, 10,
20},\ \underline{1, 4, 10, 20, 35},\ \underline{4, 10, 20, 35, 56},\ \dots
\]

\noindent{\bf Parking functions.} Parking functions are basic objects in
combinatorics closely related to~$q$-reduced divisors on~$K_n$ .  We briefly
recall these notions here. For details, see e.g.~\cite[Chapter~11]{Corry}.  A
vector~$p=(p_1,\dots,p_{n-1})\in\Z^{n-1}$ with~$1\leq p_i\leq n$ for each~$i$ is
a {\em parking function} of length~$n-1$ if for each~$j=1,\dots,n-1$,
\[
  |\left\{ i:p_i\leq j \right\}|\geq j.
\]
We partially order parking functions by~$p'\leq p$ if~$p_i'\leq p_i$ for
all~$i$.  To form all parking functions of length~$n-1$, start
with a set~$P_{n-1}$ containing the maximal parking
function~$p:=(1,\dots,n-1)$, then add all
vectors~$p'\in\Z^{n-1}$ such that~$\vec{1}\leq p'\leq p$ to~$P_{n-1}$. Finally, 
for each~$p'\in P_{n-1}$, add all vectors that arise from permuting the
coordinates of~$p'$.  The total number of parking functions of length~$n-1$
is~$n^{n-2}=|\Jac(K_n)|$.  

As with any graph, the elements of~$\Jac(K_n)$ are represented by ~$q$-reduced
divisors of degree~$0$. However, on~$K_n$ it turns out that a divisor $D$ is~$q$-reduced if and only
if~$D|_{q=0}=p-\vec{1}=(p_1-1,\dots,p_{n-1}-1)$ for some parking function $p$.  
Thus, on~$K_n$ there is a bijective correspondence between parking
functions and elements of~$\Jac(K_n)$.

We have discussed the~$\lambda$-sequence for the unique divisor class
corresponding to the smallest parking function,~$p=\vec{1}$.  We will now show
that the first differences of the~$\lambda$-sequence for any of the~$(n-1)!$
divisor classes corresponding to a maximal parking function has a particularly
nice form.  By symmetry, we may assume that~$D=(0,1,\dots,n-2,\alpha)$
where~$\alpha:=-\sum_{k=0}^{n-2}k$ so that~$\deg(D)=0$
and~$D|_{q=0}=(0,1,\dots,n-2)$.  We saw earlier that the standard
representatives of the Hermite normal form for~$\tL$
are~$k(1,\dots,1)\in\Z^{n-1}$ for~$k=0,\dots,n-1$.  Therefore, by
Proposition~\ref{prop: ps-computation}, we get~$\secondary_{[D]}$ by taking
standard representatives for elements of the following set, working modulo~$n$ in the
first~$n-1$ coordinates:
\[
  \left\{ (0,1,\dots,n-2,0)+(k(1,1,\dots,1,0):k=0,\dots,n-1\right\}.
\]
Computing the degrees of these divisors, Corollary~\ref{cor: ps-genfun} gives
\[
  \Lambda_{[D]}(z)=\frac{\sum_{i=0}^{n-1}z^{\binom{n}{2}-i}}{(1-z^n)^{n-1}(1-z)}.
\]
An analysis like that given above for~$\lambda_{[0]}$ shows that
the sequence of first differences of the~$\lambda$-sequence for~$[D]$
starts out
with~$\binom{n-1}{2}$ zeroes and then is followed by the
sequence~$\binom{k+n-2}{n-2}_{k\geq0}$ but with each term repeated~$n$ times.
For instance, on~$K_5$ we have~$D=(0,1,2,3,-6)$, and
\begin{align*}
  \lambda_{[D]}&=
  0, 0, 0, 0, 0, 0, 1, 2, 3, 4, 5, 9, 13, 17, 21, 25, 35, 45, 55, 65, 75, 95,
  115, 135, 155, 175, 210, \dots\\[5pt]
  \Delta\lambda_{[D]}&=
  0, 0, 0, 0, 0, 1, 1, 1, 1, 1, 4, 4, 4, 4, 4, 10, 10, 10, 10, 10, 20, 20,
  20, 20, 20, 35,\dots
\end{align*}

\section{Polyhedra}\label{sect: poly} 
We now interpret the results of Section \ref{sect: primsec} in terms of lattice
points in polyhedra naturally associated with divisors.

\subsection{Background} We first recall some theory, using~\cite{Beck} as our
reference.  An {\em affine $n$-cone in~$\R^n$}, or simply, an {\em $n$-cone}, is
a set of the form
\[
  \K = \left\{\,
  p+\lambda_1\omega_1+\dots+\lambda_m\omega_m:\lambda_1,\dots,\lambda_m\geq0
\right\},
\]
where~$\omega_1,\dots,\omega_m,p\in\R^n$ and the span of the~$\omega_i$ has
dimension~$n$.  The~$\omega_i$ are called {\em generators} of the
cone.  Any generator that is not a nonnegative combination of the remaining
generators is called an {\em extreme ray}. The cone is {\em pointed} if it
contains no line, and in that case~$p$ is called its {\em apex}.  We say~$\K$ is
{\em rational} if~$p,\omega_1,\dots,\omega_m\in\Q^n$, and then, by rescaling, we
may assume the~$\omega_i$ have integer coordinates.  An~$n$-cone is {\em
simplicial} if it may be written using~$n$ generators.  Simplicial cones are
necessarily pointed.

Equivalently, we may define a rational pointed $n$-cone in~$\R^n$ to be
an~$n$-dimensional intersection of finitely many half-planes of the form
\[
  \left\{ x\in\R^n: a_1x_1+\dots+a_nx_n\geq \beta\right\},
\]
where~$a_1,\dots,a_n,\beta\in\Z$ and such that the hyperplanes 
\[
  \left\{ x\in\R^n: a_1x_1+\dots+a_nx_n= \beta\right\}
\]
meet in a single point. In that case, we may express the cone as $\left\{
x\in\R^n:Ax\geq b \right\}$ where~$A$ is an integral~$m\times n$ matrix of
rank~$n$ and $b\in\Z^m$.

If~$\K$ is a simplicial~$n$-cone in~$\R^n$ with an integral generating
set~$\Omega=\{\omega_1,\dots,\omega_n\}$ and apex~$p$, define the {\em
fundamental parallelepiped} for~$\K$ with respect to~$\Omega$ to be
\[
  \textstyle
  \Pi:=\left\{\,
  p+\sum_{i=1}^n\lambda_i\omega_i:0\leq\lambda_1,\lambda_2,\dots,\lambda_n<1
\right\}.
\]
We will need the following:
\begin{property}\label{property: characterization}  Every point~$\alpha\in\K\cap\Z^n$ has a unique
  expression as
\[
\alpha = \pi+m_1\omega_1+\dots+m_d\omega_d
\]
with~$\pi\in\Pi$ and~$m_1,\dots,m_d\in\N$.
\end{property}

Define the {\em integer-point transform} of a set~$S\subset\R^n$ by
\[
  \sigma_S(\vec{z}\,)=\sigma_S(z_1,\dots,z_n):=\sum_{\alpha\in
  S\cap\Z^n}\vec{z}^{\ \alpha},
\]
where~$\vec{z}^{\ \alpha}:=\prod_{i=1}^nz_i^{\alpha_i}$.

\begin{thm}\label{thm: polyhedral generating function} {\rm (\cite[Theorem 3.5]{Beck})}  Let
  \[
    \K = \left\{\,
  p+\lambda_1\omega_1+\dots+\lambda_m\omega_n:\lambda_1,\dots,\lambda_n\geq0 \right\}
  \]
  be a simplicial~$n$-cone in~$\R^n$ with~$\omega_1,\dots,\omega_n\in\Z^n$
  and~$p\in\R^n$.  Then
  \[
    \sigma_{\K}(\vec{z}\,)=\frac{\sigma_{\Pi}(\vec{z}\,)}{\prod_{i=1}^{n}(1-\vec{z}^{\
    \omega_i})},
  \]
  where~$\Pi$ is the fundamental parallelepiped of~$\K$ with respect to the~$\omega_i$.
\end{thm}

\subsection{Linear systems and polyhedra} 
As usual, fix an ordering~$v_1,\dots,v_n$ of the vertices of~$G$ with $q=v_n$, and
then identify both~$\Div(G)$ and~$\Z^V$ with~$\Z^n$. 
\medskip

\noindent{\bf Note:} Throughout this section, we
fix the embedding
\[
  \R^{n-1}=\R^{n-1}\times\left\{ 0 \right\}\subset\R^{n}.
\]
In this way, if~$D\in\Div(G)=\Z^n$, then we may regard~$D|_{q=0}$ as an element of
either~$\Z^{n-1}$ or~$\Z^n$.  Similarly, given~$f\in\R^{n-1}$, we write~$Lf$ in place
of~$L\bigl(\begin{smallmatrix}f\\0\end{smallmatrix}\bigr)$.
\smallskip

Divisors~$D$ and~$D'$ on~$G$ are linearly
equivalent exactly when there is a function~$f\in\Z^V$ such
that~$D'=D+\smalldiv(f)$. In this context~$f$ is referred to as a {\em firing
script}, and we express the complete linear system for~$D$ as
\[
  |D|=\left\{ E\in\Div(G): E = D+Lf\geq0 \text{ for some firing
  script~$f$}\right\}.
\]
The set of firing scripts appearing above for the complete linear system for~$D$
form the polyhedron
\[
  Q_D :=\left\{f\in\R^n: Lf\geq -D\right\}\subset\R^n.    
\]
However, the integer points of~$Q_D$ are not in bijection with elements of~$|D|$
since~$L$ has a non-trivial kernel.  The kernel is generated by the
all-ones vector~$\vec{1}$; so modulo~$\ker(L)$, each firing
script~$f=(f_1,\dots,f_n)$ has the unique
representative $f-f_n\cdot\vec{1}$ with last coordinate~$0$,
leading us to define
\[
  P_D:=Q_D\cap\left\{f\in\R^n: f_n=0 \right\}\subset\R^{n-1}
\]
so that~$Q_D=P_D+\R\vec{1}\subset\R^n$.
It is straightforward to see that the integer points~$P_D\cap\Z^{n-1}$ are in
bijection with~$|D|$:
\begin{equation}\label{eqn: polytope bijection}
  f\in P_D\cap\Z^{n-1}\quad\longleftrightarrow\quad
  D+Lf\in|D|.
\end{equation}
Since~$|D|$ is finite, it follows that the polyhedron~$P_D$ is bounded, and hence is
a~{\em polytope}. (For a direct proof of boundedness, see
\cite[Proposition~2.20]{Corry}.)

If~$D\sim D'$ with~$D'=D+Lf$, then the polyhedra associated with these divisors
differ by a translation: $Q_{D}=f+Q_{D'}$, and as discussed above, we may
assume~$f_n=0$ to write $P_{D}=f+P_{D'}$.

The ideas presented above may be applied in order to
characterize~$\E_{[D]}=\cup_{k\geq0}|D+kq|$ in terms of firing vectors.
\begin{defn} The {\em $q$-cone} for a divisor~$D\in\Div^0(G)$ is the set
\[
  \K_D := \left\{ (f,t)\in\R^n\times\R : Lf+tq\geq -D \text{ and } f_n=0\right\}
\subset\R^{n-1}\times\R.
\] 
\end{defn}

\Needspace{3\baselineskip}
\begin{thm}\label{thm: poly} Let~$D\in\Div^0(G)$. 
Then $\K_D$ is a rational simplicial~$n$-cone with
      apex~$p:=\tL^{-1}(-D|_{q=0})\in\R^{n-1}\times\left\{ 0
      \right\}\subset\R^{n}$ and has the following properties: 
  \begin{enumerate}[label=\rm{(\arabic*)},leftmargin=*]
    \item\label{poly1} The set of integer points of~$\K_D$ is in bijection
      with~$\E_{[D]}$ via the mapping
      \begin{align*}
	\psi_D\colon\K_D\cap\Z^{n}&\xrightarrow{\sim}\E_{[D]}\\
	(f,k)\ &\mapsto D+kq+Lf.
      \end{align*}

    \item\label{poly2} The mapping~$\psi_0$ restricts to a bijection between
      generating sets of integral extreme rays for~$\K_D$ and sets of primary
      divisors for~$G$.  Let~$\Omega$ be a generating set of integral extreme
      rays for~$\K_D$ with corresponding set of primary
      divisors~$\primary=\psi_0(\Omega)$.  Let~$\Pi$ be the corresponding
      fundamental parallelepiped.
      Then $\psi_D$ restricts to a bijection between the integer
      points of~$\Pi$ and the secondary divisors of $[D]$ with respect to
      $\primary$.  
    \item\label{poly3} Let~$\Omega$ and~$\Pi$ be as in part~\ref{poly2}.  Then the~$\lambda$-sequence generating function for~$\E_{[D]}$ is
      \[
	\Lambda_{[D]}(z)=\sigma_{\K}(1,\dots,1,z)=\frac{\sigma_{\Pi}(1,\dots,1,z)}{\prod_{\omega\in\Omega}(1-z^{\deg(\omega)})},
      \]
      where~$\deg(\omega)$ is the sum of the coordinates of~$\omega$.  The numerator and
      denominator of the expression on the right are the same as those appearing
      in Corollary~\ref{cor: ps-genfun}.
  \end{enumerate}
\end{thm}

\begin{proof} Let~$\tilde{r}_n\in\R^{n-1}$ denote the last row of~$L$ with its
  last entry removed. Then
\begin{equation}
  \label{eqn: K_D}
  \K_D = \left\{ (f,t)\in\R^{n-1}\times\R: \tL f\geq -D|_{q=0}\text{ and
  }\tilde{r}_n\cdot f+t\geq-D(q)\right\}.  
\end{equation}
Since~$\tL$ is invertible, these defining conditions are independent, and it
follows that~$\K_D$ is a rational~$n$-cone. To find the apex, first
solve~$\tL f=-D|_{q=0}$ to find~$f=\tL^{-1}(-D|_{q=0})$.  Next, since the sum of the
  rows of~$L$ is~$0$, if follows that~$\tilde{r}_n\cdot f=-\tL f$, and it is now
  easy to verify that the last coordinate of the apex is~$t=0$ using the fact
  that~$\deg(D)=0$.

The rest follows immediately from the discussion preceding the theorem.
Part~\ref{poly1} uses the fact that every firing script has a unique
representative modulo $\ker L$ having final coordinate~$0$.  Part~\ref{poly2}
relies on Property~\ref{property: characterization}.  Part~\ref{poly3} follows
since~$\deg\psi_D(f,k)=k$.
\end{proof}


The following proposition shows that the essential information encoded
in~$\K_{D}$ is contained in its bottom (with respect to the last
coordinate) face:
\begin{prop}\label{prop: facet}  Given~$D\in\Div^0(G)$, take the union of the
  nested sequence of
  polytopes~$P_D\subset P_{D+q}\subset P_{D+2q}\subset\dots$ to define
  \[
    \tK_D:=\bigcup_{k\in\N} P_{D+kq}\subset\R^{n-1}.
  \]
  \begin{enumerate}[label=\rm{(\arabic*)},leftmargin=*]
    \item\label{facet1} $\tK_D$ is a rational simplicial~$(n-1)$-cone, and
      \[
	\tK_D=\left\{ f\in\R^{n-1}:\tL f\geq-D|_{q=0} \right\}.
      \]
      The apex of~$\tK_D$ is~$\tilde{p}:=\tL^{-1}(-D|_{q=0})\in\R^{n-1}$.
    \item\label{facet2} Let~$\tilde{r}_n\in\R^{n-1}$ be the last row of the Laplacian
      matrix with its final entry removed.  Then there is a
      injection
      \begin{align*}
	\imath_D\colon\tK_D&\to\K_D\subset\R^{n-1}\times\R\\
	f&\mapsto (f,-D(q)-(\tilde{r}_n\cdot f)).
      \end{align*}
      The image of~$\imath_D$ is the facet of~$\K_D$ which is the intersection
      of~$K_D$ with the hyperplane 
      \[
	H:=\left\{(f,t)\in\R^{n-1}\times\R: (\tilde{r}_n\cdot f)+t=-D(q)  \right\},
      \]
      and
      \[
	\K_D=\imath_D(\tK_D)+\R_{\geq0}q.
      \]
    \item\label{facet3} Write
      \[
	\tK_D=\left\{\tilde{p}+\lambda_1\tilde{\omega}_1
	  +\dots+\lambda_{n-1}\tilde{\omega}_{n-1}:
	  \lambda_1,\dots,\lambda_{n-1}\geq0
	\right\}
      \]
      with integral generating
      set~$\tilde{\Omega}:=\{\tilde{\omega}_1,\dots,\tilde{\omega}_{n-1}\}\subset\Z^{n-1}$.
      Let~$\tilde{v}_i$ denote the~$i$-th standard basis vector for~$\R^{n-1}$.
      Then up to re-indexing, $\tilde{\omega}_i=\tL^{-1}(\tl_{i}\tilde{v}_i)$
      where~$\tl_{i}$ is a positive integer multiple of~$\ord_q(v_i)$.  The set
      $\Omega:=\imath_0(\tilde{\Omega})\cup\left\{ \tl_qq \right\}$ is an
      integral generating set of extreme rays for~$\K_D$ for any choice of
      positive integer~$\tl_q$.  Every integral generating set of extreme rays
      for~$\K_D$ arises in this manner.  With this notation,
      let~$\widetilde{\Pi}$ and~$\Pi$ be the fundamental parallelepipeds
      for~$\tilde{\Omega}$ and~$\Omega$, respectively.  Then
      \[
	\Pi\cap\Z^n=\left\{ \imath_D(\tilde{\pi})+\ell q:
	\tilde{\pi}\in\widetilde{\Pi}\cap\Z^{n-1}\text{ and } 0\leq \ell<\tl_q\right\}.
      \]
  \end{enumerate}
\end{prop}
\begin{proof}  We have~$f\in\tK_D$ if and only if
  \[
    \tL f\geq-D|_{q=0}\quad\text{and}\quad \tilde{r}_n\cdot f\geq -D(q)-k
  \]
  for some~$k\in\N$.  The second condition is superfluous since~$k$ can
  be arbitrarily large.  The fact that~$\tK$ is a rational
  simplicial~$(n-1)$-cone with apex~$\tL^{-1}(-D|_{q=0})$ follows since~$\tL$ is 
  invertible.  This establishes part~\ref{facet1}.
  Part~\ref{facet2} then follows from part~\ref{facet1} and the description
  of~$\K_D$ given in~\eqref{eqn: K_D} in the proof of Theorem~\ref{thm: poly}.
  Since~$K_D=\imath_D{\tK_D}+\R_{\geq0}q$, part~\ref{facet3} follows from
  Theorem~\ref{thm: poly}~\ref{poly2} and Theorem~\ref{thm:
  primary-secondary}~\ref{ps-2}.
\end{proof}

\begin{example}\label{example: cone} Let graph~$G=C_3=K_3$ with vertex set~$v_1,v_2$, and $v_3=q$,
  and consider the divisor~$D=(1,0,-1)$ of degree~$0$.
  The Laplacian matrix is
  \[
L=
\left(\begin{array}{rrr}
    2&-1&-1\\
    -1&2&-1\\
    -1&-1&2
\end{array} \right).
  \]
  The cone~$\tK_D$ is defined by the system of inequalities
  \begin{equation*}
    \sysdelim..
    \syslineskipcoeff{1.3}
    \systeme{
      2x_1-x_2\geq-1,
      -x_1+2x_2\geq\hphantom{-}0.
    }
  \end{equation*}
To find generators for~$\tK_D$ and~$\K_D$, we use Proposition~\ref{prop:
facet}~\ref{facet3}.  We have
  \[
    \tL^{-1}=
	\frac{1}{3}
    \left(\begin{array}{cc}
	2&1\\
	1&2
    \end{array} \right).
  \]
  Multiply the first column of~$\tL^{-1}$ by~$\ord_q(v_1)=3$ to
  get~$\tilde{w}_1=(2,1)$.  Multiply the second column by~$\ord_q(v_2)=3$ to
  get~$\tilde{w}_2=(1,2)$.  Every set of integral extreme rays for~$\tK_D$ will be positive integer
  multiples of these.  Since~$D|_{q=0}=(1,0)$, the apex of the cone is
  \[
    \tL^{-1}(-D|_{q=0})=(-2/3,-1/3).
  \]
  Thus, 
  \[
    \tK_D=\left\{ (-2/3,-1/3)+\lambda_1(2,1)+\lambda_2(1,2) \right\}.
  \]
  Using the notation of Proposition~\ref{prop: facet}~\ref{facet2}, we
  have~$\tilde{r}_n=(-1,-1)$.  Therefore, taking~$\tl_q=1$, we get the set of
  extreme rays for~$\K_D$:
  \[
    \Omega=\left\{\imath_0(\tilde{\omega}_1),\imath_0(\tilde{\omega}_2),q \right\}
    =\left\{ (2,1,3),(1,2,3),(0,0,1) \right\},
  \]
  and
  \[
    \K_D=\left\{
      (-2/3,-1/3,0)+\lambda_1(2,1,3)+\lambda_2(1,2,3)+\lambda_3(0,0,1)
    \right\}.
  \]
  The cone~$\tK_D$ is pictured in Figure~\ref{fig: cone} along with its
  fundamental parallelogram~$\widetilde{\Pi}$ with respect
  to~$\{\tilde{\omega}_1,\tilde{\omega}_2\}$.
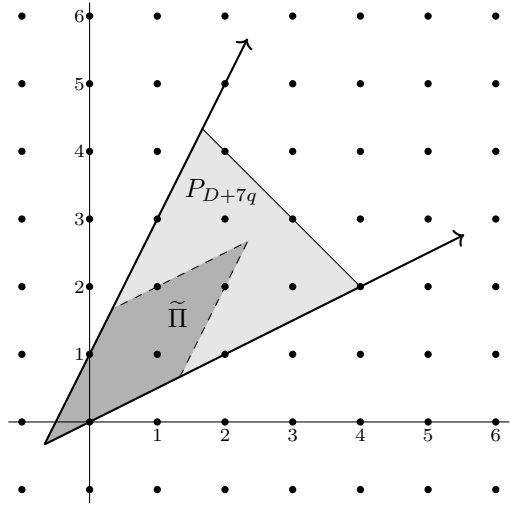
\begin{figure}[ht]
\centering
\begin{tikzpicture}[scale=0.9]
  \filldraw[color=gray!20] (-2/3,-1/3)--(5/3,13/3)--(4,2)--(-2/3,-1/3);
  \draw[thin] (5/3,13/3)--(4,2);
    \filldraw[color=gray!60]
    (-2/3,-1/3)--(-2/3+2,-1/3+1)--(-2/3+3,-1/3+3)--(-2/3+1,-1/3+2)--(-2/3,-1/3);
    \draw[dashed] (-2/3+2,-1/3+1)--(-2/3+3,-1/3+3)--(-2/3+1,-1/3+2);
  \draw (-1.2,0)--(6.2,0);
  \draw (0,-1.2)--(0,6.2);
  \foreach \x in {1,2,...,6}
  {
    \node at (-0.15,\x) {$\scriptstyle\x$};
    \node at (\x,-0.2) {$\scriptstyle\x$};
  }
  \foreach \i in {-1,0,...,6}
    \foreach \j in {-1,0,...,6}
    {
      \filldraw (\i,\j) circle(1.3pt);
    }
    \def\sone{3.1}
    \def\stwo{3.0}
    \draw[->,thick] (-2/3,-1/3)--(-2/3+\sone*2,-1/3+\sone*1);
    \draw[->,thick] (-2/3,-1/3)--(-2/3+\stwo*1,-1/3+\stwo*2);
    \node at (1.3,1.6) {$\widetilde{\Pi}$};
    \node at (1.95,3.4) {$P_{D+7q}$};
\end{tikzpicture}
\caption{The cone~$\tK_D$, a fundamental parallelogram, and the
polytope~$P_{D+7q}$ for the divisor~$D=(1,0,-1)$ on~$C_3$
(cf.~Example~\ref{example: cone}).}\label{fig: cone}
\end{figure}

There are three integer points in~$\widetilde{\Pi}$:
\[
  \widetilde{\Pi}\cap\Z^2=\left\{ (0,0),(0,1), (1,1) \right\},
\]
and the integer points of~$\Pi$ are just the ``lifts'' of these via~$\imath_D$:
\[
  \Pi\cap\Z^3=\left\{ (0,0,1),(0,1,2), (1,1,3) \right\}.
\]
Thus, the integer-point transform is
\[
  \sigma_{\K}(z_1,z_2,z_3)
  =\frac{z_3+z_2z_3^2+z_1z_2z_3^3}{(1-z_1^2z_2z_3^3)(1-z_1z_2^2z_3^3)(1-z_3)}.
\]
The~$\lambda$-sequence generating function is therefore
\begin{align*}
  \Lambda_{[D]}(z)&=\sigma_{\K}(1,1,z)=\frac{z+z^2+z^3}{(1-z^3)^2(1-z)}\\[8pt]
  &=z + 2 z^{2} + 3 z^{3} + 5 z^{4} + 7 z^{5} + 9 z^{6} + 12 z^{7} + 15 z^{8} +
  18 z^{9} + 22 z^{10} + 26 z^{11}+\dots
\end{align*}
This is exactly what we get from Corollary~\ref{cor: ps-genfun} using the
primary and~$[D]$-secondary divisors
\begin{align*}
  \primary &= \psi_0(\Omega)=\left\{ (3,0,0),(0,3,0),(0,0,1 \right\}\\
    \secondary_{[D]} &= \psi_D(\Pi\cap\Z^{n})=\left\{ (1,0,0),(0,2,0),(2,1,0)
    \right\}.
\end{align*}
(In Example~\ref{example: ps-complete} we calculated $\secondary_{[0]}=\left\{
(0,0,0),(1,1,1),(2,2,2)\right\}$ for $K_3=C_3$.
Proposition~\ref{prop: ps-computation}~\ref{ps-comp2} then says
$\secondary_{[D]}$ consists of standard representatives
for~$D+\secondary_{[0]}$ in ~$\Z/3\Z\times\Z/3\Z\times\Z/1\Z$, which
agrees with the above computation.) 

The generating function~$\Lambda_{[D]}(z)$ predicts, for example, that there
are~$12$ elements in~$|D+7q|=|(1,0,6)|$.  The polytope~$P_{D+7q}$ is defined by the system of
inequalities
  \begin{equation*}
    \sysdelim..
    \syslineskipcoeff{1.2}
    \systeme{
      2x-y\geq-1,
      x-2y\geq\hphantom{-}0,
      -x-y\geq-6.
    }
  \end{equation*}
In Figure~\ref{fig: cone}, we can see the~$12$ lattice points in~$P_{D+7q}$
corresponding to the elements of~$|D+7q|$.
\end{example}

\section{Invariant theory}\label{sect: molien}
The results in Section \ref{sect: primsec} may also be interpreted in terms of
the invariant theory for a representation of the dual group $\Jac^*(G)$. Through
this lens, primary and secondary divisors become primary and secondary
invariants, and $\Lambda_{D}(z)$ is given a substantially different expression
as a Molien series. 
\subsection{Background}
We first recall basic invariant theory for finite groups with~\cite{Stanley}
and\cite{Sturmfels} as references.  Given a matrix~$A\in\GL(\C^n)$ and a
polynomial~$f\in\polyring:=\C[x_1,\dots,x_n]$, define~$f\circ A$ by
\[
  (f\circ A)(x_1,\dots,x_n)=f(A\vec{x})
\]
where~$\vec{x}$ is the column vector~$[x_1,\dots,x_n]^t$.  Given a finite
subgroup~$\Gamma$ of~$\GL(\C^n)$ and a
character~$\chi\colon\Gamma\to\C^{\times}:=\C\setminus\left\{ 0 \right\}$,
define the {\em $\chi$-relative invariants of~$\Gamma$} to be elements of 
\[
  \polyring^{\Gamma}_{\chi}:=\left\{ f\in\polyring: f\circ\gamma=\chi(\gamma)f
  \text{ for all~$\gamma\in\Gamma$} \right\}.
\]
The~{\em $\chi$-relative Reynolds operator} is defined for each
polynomial~$f\in\polyring$ by
\[
  \mathcal{R}_{\chi}(f)=\frac{1}{|G|}\sum_{\gamma\in\Gamma}\overline{\chi}(\gamma)f\circ\gamma.
\]
It is easy to check that~$\mathcal{R}_{\chi}$ is linear in~$f$ and that~$f$
is~$\chi$-invariant if and only if~$\mathcal{R}_{\chi}(f)=f$.  In the
case~$\chi=\varepsilon$, the trivial character,
$\polyring^{\Gamma}:=\polyring^{\Gamma}_{\varepsilon}$ is a subring
of~$\polyring$, graded by degree, called the {\em invariant subring
of~$\Gamma$}.  It is generated by~$\mathcal{R}_{\varepsilon}(f)$ as~$f$ ranges
over all monomials of degree at most~$|\Gamma|$. The elements
of~$\polyring^{\Gamma}$ are simply called~{\em invariants of~$\Gamma$}
and~$\mathcal{R}:=\mathcal{R}_{\varepsilon}$ is the~{\em Reynolds operator
for~$\Gamma$}.  For arbitrary~$\chi$, the relative
invariants~$\polyring^{\Gamma}_{\chi}$ form a
graded~$\polyring^{\Gamma}$-module, generated by the homogeneous
polynomials~$\mathcal{R}_{\chi}(f)$ as~$f$ ranges over all monomials of degree
at most~$|\Gamma|$. 

There exist algebraically independent homogeneous
invariants~$p_1,\dots,p_n$ such that~$\polyring^{\Gamma}$ is a
finitely-generated free module over $\C[p_1,\dots,p_n]$.
For any character~$\chi$, if~$q_1,\dots,q_t$ are homogeneous polynomials
forming a~$\C$-basis for~$\polyring^{\Gamma}_{\chi}$ modulo the
submodule $\sum_{i=1}^np_i\polyring^{\Gamma}_{\chi}$, then
\[
  \polyring_{\chi}^{\Gamma}=\bigoplus_{i=1}^{t}q_i\C[p_1,\dots,p_n].
\]
The~$p_i$ are called {\em primary invariants} and are independent of~$\chi$.
The~$q_i$ are called {\em secondary (relative) invariants} and depend
on~$\chi$. The number of secondary invariants,~$t$, also
depends on~$\chi$ in general.  However, letting~$t_{\varepsilon}$ be the number of
secondary invariants for the trivial character, we have
\[
  t_{\varepsilon}|\Gamma|=\prod_{i=1}^n\deg(p_i).
\]

The {\em Hilbert series} for~$\polyring^{\Gamma}_{\chi}$ is
\[
  \Phi_{\Gamma,\chi}(z):=\sum_{d\geq
  0}\dim_{\C}(\polyring_{\chi,d}^{\Gamma})z^d,
\]
where~$\polyring_{\chi,d}^{\Gamma}$ denotes the~$d$-th graded piece
of~$\polyring_{\chi}^{\Gamma}$.
The Hilbert series is also known as the (relative) {\em Molien series} for~$\Gamma$ due to a
theorem of Molien which states that
\begin{equation} \label{eqn: Molien}
  \Phi_{\Gamma,\chi}(z) =
  \frac{1}{|\Gamma|}\sum_{\gamma\in\Gamma}\frac{\overline{\chi(\gamma)}}{\det(I_n-z\gamma)}.
\end{equation}

\subsection{Linear systems}\label{subsection: linear systems} Order the vertices~$v_1,\dots,v_n$ of~$G$, and
fix~$q=v_n$. To see the relevance of invariant theory to our problem, start
with the sequence of projections
\begin{center}
\begin{tikzcd}[row sep = 1pt ]
\Z^n=\Div(G)\arrow[r,twoheadrightarrow]&\Pic(G)\arrow[r,twoheadrightarrow]&\Jac(G)\\
 \qquad  D\arrow[r,mapsto,shorten >= 14pt,shorten <= 14pt,xshift=2pt ]
 &\left[D\right]\arrow[r,mapsto,shorten >= 10pt,shorten <= 10pt,xshift=2pt ]
 &\left[D-\deg(D)q\right].
\end{tikzcd}
\end{center}
Apply the functor $\Hom(\,\cdot\,,\C^{\times})$ to get a sequence of dual groups
\[
  \Jac(G)^*\hookrightarrow\Pic(G)^*\hookrightarrow(\C^\times)^n\subset GL(\C^n),
\]
identifying~$(\C^{\times})^n$ with diagonal matrices having nonzero diagonal
entries.  Define~$\rho$ to be the composition of these mappings:
\begin{align}\label{rho}
  \rho\colon\Jac(G)^{*}\ &\xrightarrow{\qquad}\qquad\GL(\C^n)\\
  \nonumber
  \chi\qquad
  &\mapsto\quad\diag(\chi([v_1-q]),\chi([v_2-q]),\dots,\chi([v_{n-1}-q]),1).
\end{align}
\begin{thm}\label{thm: invariants}
  Consider each~$[D]\in\Jac(G)$ as a character of~$\Gamma:=\im(\rho)\subset\GL(\C^n)$
  via~$[D](\rho(\chi)):=\chi([D])$ for each~$\chi\in\Jac(G)^{*}$.  Then 
  \begin{enumerate}[label=\rm{(\arabic*)},leftmargin=*]
    \item\label{inv1} 
\[
  \textstyle \left\{x^{E}:=\prod_{i=1}^nx_i^{E(v_i)}: E\in\E_{[D]} \right\}
\]
is a~$\C$-basis for the relative invariants~$\polyring^{\Gamma}_{[D]}$,
\item\label{inv2} $\polyring=\bigoplus_{[D]\in\Jac(G)}\polyring^{\Gamma}_{[D]}$, and
\item\label{inv3}  the correspondence~$E\mapsto x^{E}$ for effective divisors~$E$ gives a
  bijection between systems of primary and~$[D]$-secondary divisors and systems
  of monomial primary and~$[D]$-relative invariants.
  \end{enumerate}
\end{thm}
\begin{proof} An arbitrary element of~$\polyring$ may be written
  as~$f=\sum^{\,}_{E}a_Ex^E$ where the sum is over all effective divisors of~$G$
  and all but finitely many~$a_E$ are zero.  Let~$\chi\in\Jac(G)^{*}$.  Then
  \allowdisplaybreaks
  \begin{align*}
    f\circ\rho(\chi)
    &= \sum_{E}a^{\,}_E(x^E\circ\rho(\chi))\\
    &= \sum_{E}a^{\,}_E\textstyle\left(\prod_{v\in V}\chi([v-q])^{E(v)} \right)x^E\\
    &= \sum_{E}a^{\,}_E\textstyle\,\chi\!\left(\sum_{v\in V}E(v)([v-q])\right)x^E\\
    &= \sum_{E}a^{\,}_E\textstyle\,\chi\!\left([E-\deg(E)q]\right)x^E.
  \end{align*}
  Therefore,~$f\in\polyring^{\Gamma}_{[D]}$ if and only if for each~$E$ such
  that~$a^{\,}_E\neq0$, we have $\chi([E-\deg(E)q])=\chi([D])$ for all~$\chi$,
  or equivalently, $[E-\deg(E)q]=[D]$, i.e.,~$E\in\E_{[D]}$.
 Parts~\ref{inv1},~\ref{inv2}, and~\ref{inv3} follow.  Part~\ref{inv2} reflects
 the fact that the~$\E_{[D]}$ partition the set of effective divisors.
\end{proof}
As an immediate corollary, we may express $\lambda$-sequence
generating functions as a Molien series. These expressions differ from those given
 in Corollary~\ref{thm: primary-secondary} and Theorem~\ref{thm:
polyhedral generating function} (which are identical to each other).
\begin{cor}\label{cor: lambda-molien}
  Let~$[D]\in\Jac(G)$.  The generating function for the $\lambda$-sequence
  for~$\E_{[D]}$ is given by the Molien series
  \[
    \Lambda_{[D]}(z)=\Phi_{\Gamma,[D]}(z)
    =\frac{1}{|\Jac(G)|}\sum_{\chi\in\Jac(G)^{*}}\frac{\overline{\chi([D])}}{\det(I_n-z\rho(\chi))}.
  \]
\end{cor}

To compute with Corollary~\ref{cor: lambda-molien} concretely,  use
integer row and column operations to reduce~$\tL$ to diagonal form (e.g., Smith
normal form), denoting the result by~$M:=\diag(m_1,\dots,m_{n-1})$.  Record the
row and column operations in matrices~$U$ and~$W$ so that~$U\tL W = M$.
Then~$U$ descends to an isomorphism
\[
  \Jac(G)\simeq\Z^{n-1}/\im_{\Z}\tL\xrightarrow{\ U\ }\prod_{i=1}^{n-1}\Z/m_i\Z.
\]
Having identified~$\Jac(G)$ with~$R:=\prod_{i=1}^{n-1}\Z/m_i\Z$, we now describe
the characters.  For each~$r\in R$, choose a representative
lifting~$(r_1,\dots,r_{n-1})\in\Z^{n-1}$, and let $ \tilde{r} :=
\left(\frac{r_1}{m_1},\dots,\frac{r_{n-1}}{m_{n-1}} \right)$.  Define the
character~$\tilde{\chi}_r\in R^*$ by
\[
  \tilde{\chi}_r(a) :=\exp(2\pi i\,(\tilde{r}\cdot a))
\]
for each~$a\in R$.  Then define~$\chi_r\in\Jac^*(G)$
by
\[
  \chi_r([D]):=\tilde{\chi}_r(U(D|_{q=0}))
\]
for each~$[D]\in\Jac(G)$.  It follows that
\[
  \rho(\chi_r)=\diag(\tilde{\chi}_r(u_1),\tilde{\chi}_r(u_2),\dots,\tilde{\chi}_r(u_{n-1}),1)
\]
where~$u_j$ is the~$j$-th column of~$U$.  In all of the above, if~$m_k=1$ for
some~$k$, then the~$k$-th factor of~$R$ and the~$k$-th {\em row} of~$U$ may be
dropped.

\subsection{Examples}  The following examples use Corollary~\ref{cor:
lambda-molien} to compute~$\lambda$-generating functions.  For a direct
application of the relation between polynomial invariants and linear systems
exhibited in Theorem~\ref{thm: invariants}, see Section~\ref{sect: necklaces}.
\Needspace{3\baselineskip}
\subsubsection{Trees}  If~$G$ is a tree, then~$\Jac(G)$ is trivial,
and~$\Jac(G)^*$ contains only the trivial character.  So
Corollary~\ref{cor: lambda-molien} says~$\Lambda_{[0](z)}=1/(1-z)^{n}$.

\subsubsection{Diamond graph} Now let~$G$ be the diamond graph of
Figure~\ref{fig: diamond graph}.  Letting
\[
U = 
\left(\begin{array}{rrr}
    0 & 0 & 1 \\
    0 & 1 & 0 \\
    1 & -1 & -4
\end{array}\right)
\quad\text{and}\quad
W = 
\left(\begin{array}{rrr}
    3 & 1 & 5 \\
    2 & 1 & 3 \\
    3 & 1 & 4
\end{array}\right),
\]
we have~$U\tL W =\diag(1,1,8)$ and the corresponding isomorphism
\begin{align*}
  \Jac(G)\simeq\Z^{3}/\im_{\Z}\tL&\to\Z/8\Z\\
  (a,b,c)&\mapsto a-b-4c.
\end{align*}
The divisor class~$[v_1-q]$ generates~$\Jac(G)$. Let~$D_j=j[v_1-q]$
for~$j=0,\dots,7$, and let $\omega=\exp(2\pi i/8)$.  Then, by
Corollary~\ref{cor: lambda-molien}, 
\[
  \Lambda_{[D_j]} =
  \frac{1}{8(1-z)}\sum_{k=0}^7\frac{\omega^{-jk}}{(1-\omega^{k}z)(1-\omega^{-k}z)(1-\omega^{-4k}z)}
\]
for~$j=0,\dots,7$.
\subsubsection{Cycle graphs} Now let~$G=C_n$ be a cycle graph using the notation
of Example~\ref{example: ps-cycle}. Let~$I_n$ be the~$n\times n$ identity matrix,
and let~$U$ be the matrix formed by replacing
the last row of the~$-I_n$ with the row
$(1,2,\dots,(n-2),-1)$.  Let~$W$ be the~$n\times n$ matrix with
$W_{ij}=-\min\left\{ i,j \right\}$ (so~$W$ starts with a row of~$-1$s and ends
  with the row $(-1,-2,\dots,-(n-1))$).  Then~$U\tL W=\diag(1,1,\dots,1,n)$,
  and multiplication by~$U$ gives the isomorphism
\begin{align*}
  \Jac(C_n)\simeq\Z^{n-1}/\im_{\Z}\tL&\to\Z/n\Z\\
  (a_1,\dots,a_{n-1})&\mapsto a_1+2a_2+\dots+(n-2)a_{n-2}+(n-1)a_{n-1}.
\end{align*}
The divisor class~$[v_{1}-q]$ generates~$\Jac(C_n)$. Let $D_j=j[v_{1}-q]$
for~$j=0,\dots,n-1$, and let~$\omega=\exp(2\pi i/n)$.  Then by
Corollary~\ref{cor: lambda-molien},
\[
  \Lambda_{[D_j]}(z)
  =
  \frac{1}{n}\sum_{k=0}^{n-1}\frac{\omega^{-jk}}{\prod_{t=0}^{n-1}(1-\omega^{tk}z)}
\]
for~$j=0,\dots,n-1$.  In particular,
\[
  \Lambda_{[0]}(z)=\frac{1}{n}\sum_{k=0}^{n-1}\frac{1}{(1-z^{n/\gcd(n,k)})^{\gcd(n,k)}}
  =\frac{1}{n}\sum_{d|n}\frac{\phi(d)}{(1-z^d)^{n/d}}
\]
where~$\phi$ is the Euler totient function. We shall explore this example
further in Section~\ref{sect: necklaces}.  

\subsubsection{Complete graphs} Let~$G=K_n$.  Perform integer column operations
to bring~$\tL$ into Hermite normal form~$H$ as described in Example~\ref{example:
ps-complete}. Next, let~$U$ be the matrix formed by replacing the first column
of the identity matrix~$I_{n-1}$ by the column $(1,-1,-1,\dots,-1)$.
Then~$UH = M:=\diag(1,n,n,\dots,n)$.  Since~$M_{1,1}=1$, alter~$U$ by removing
its first row, and we get an
isomorphism
\[
  \Jac(K_n)\simeq\Z^{n-1}/\im_{\Z}\tL\xrightarrow{U}R:=\left(\Z/n\Z  \right)^{n-2}.
\]
Let~$\omega:=\exp(2\pi i/n)$.  Then for each~$r\in R$, we have the
character~$\chi_r$ such that~$\chi_r(a)=\omega^{r\cdot a}$ for
each $a\in\Z^{n-1}/\im_{\Z}\tL$.  By Corollary~\ref{cor: lambda-molien}, for
each~$[D]\in\Jac(G)$, writing~$D=(d_1,\dots,d_n)$, 
\[
  \Lambda_{[D]}(z)=\frac{1}{n^{n-2}(1-z)}\sum_{r\in
  (\Z/n\Z)^{n-2}}\frac{\omega^{-r\cdot(d_2-d_1,\dots,d_{n-1}-d_1)}}{(1-\omega^{-r_1-\dots-r_{n-2}}z)(1-\omega^{r_1}z)(1-\omega^{r_2}z)\dots(1-\omega^{r_{n-2}}z)}.
\]

\begin{remark} Corollaries~\ref{cor: ps-genfun} and~\ref{cor: lambda-molien}
  give two ways of expressing the generating function~$\Lambda_{[D]}(z)$.  One
  sums over elements of~$\secondary_{[D]}$, and the other sums over elements
  of~$\Jac^*(G)$.  In practice, one of these two expressions may be much simpler
  than the other.  For instance, the complete graph~$K_n$ has a large Jacobian
  group,~$|\Jac(K_n)|=n^{n-2}$, however we can find a set of secondary divisors
  with only~$n$ elements.  So the expression for~$\Lambda_{[D]}(z)$ coming from
  Corollary~\ref{cor: lambda-molien} will have~$n^{n-2}$ summands while the
  numerator~$S$ appearing in Corollary~\ref{cor: ps-genfun} with have only~$n$
  terms.  In the case of the cyclic graph~$C_n$, we have the
  opposite situation:~$|\Jac(C_n)|=n$ and there are~$n^{n-2}$ secondary divisors
  (taking~$\tl_v=n$ for all vertices~$v\neq 1$).
\end{remark}

\section{Cycle graphs and necklaces}\label{sect: necklaces}

\subsection{Necklaces} Let~$C$ be a finite set of~{\em colors} and let~$C^n$
denote the set of all words (strings) of length~$n$ with letters in~$C$.
Let~$\sigma$ be the cyclic shift operator on~$C^n$:
\[
  \sigma(c_1\dots c_n) = c_n c_1\dots c_{n-1}.
\] 
Define an equivalence relation on~$C^n$ by letting~$w\sim w'$
if~$w=\sigma^i(w')$ for some integer~$i$.  A {\em necklace} of length~$n$ on the
color set~$C$ is an equivalence class $[c_1\dots c_n]\in C^n/\!\!\sim$.  We
think of each~$c_i$ as being a bead of color~$c_i$.  The~{\em period} of a
necklace~$N=[w]$ is the smallest positive integer~$i$ such that~$\sigma^i(w)=w$.

\begin{defn} Let~$m$ be a positive integer. A necklace is {\em
$m$-divisible} if its period is divisible by~$m$. (See Figure~\ref{fig:
necklaces} for an example.) 
\end{defn}

A {\em binary necklace} is a necklace for which~$C$ consists of two colors
(which we take to be black and white).  Let~$\neck(n,k)$ denote the set
of binary necklaces with~$n$ black beads and~$k$ white beads.  

\begin{defn} Let~$N=[bw^{a_1}bw^{a_2}\cdots bw^{a_n}]\in\neck(n,k)$.  The
  {\em code} for~$N$ is the necklace with~$n$ beads and with colors~$\left\{
  a_1,\dots,a_n \right\}\subset\N$,
  \[
    \mathrm{code}(N):=[a_1\dots a_n].
  \]
\end{defn}

\subsection{Linear systems on cycle graphs} We use the notation of
Section~\ref{example: ps-cycle}.  Let~$C_n$ be the cycle graph with vertices
$v_1,\dots,v_n$ around the cycle, and take~$q:=v_n$.  Working with subscripts
modulo~$n$, let~$D_i=v_i-q\in\Div(C_n)$ for all~$i\in\Z$.  Then~$\Jac(C_n)$ is the
cyclic group of order~$n$ generated by~$[D_1]$ with~$j[D_1]=[D_j]$ for all~$j$. The dual
group~$\Jac(C_n)^*$ is generated by the character~$\chi_1$ determined
by $\chi_1([D_1])=\omega$ where~$\omega$ is a primitive~$n$-th root of unity.  The
representation~$\rho\colon\Jac(G)^*\to\GL(\C^n)$ described in
Section~\ref{subsection: linear systems} is
determined by
\[
  \rho(\chi_1)=\diag(\omega,\omega^2,\dots,\omega^n).\footnote{We often
  write~$\omega^n$ instead of~$1$ for consistency of notation.}
\]
Changing coordinates on~$\C^n$, we conjugate this diagonal representation
into a permutation representation.  In detail, for any~$x\in\C$,
let~$v(x):=(x,x^2,\dots,x^n)$ and let~$A$ be the matrix with rows
$v(\omega^n),v(\omega),v(\omega^{2}),\dots,v(\omega^{n-1})$.  We have
$A^{-1}=\frac{1}{n}\,\overline{A}^{\,t}$. Conjugate by~$A$ to
get~$\rho':\Jac(G)^*\to\GL(\C^n)$ where $\rho'(\chi)=A\rho(\chi)A^{-1}$ for
all~$\chi\in\Jac^*(C_n)$.  Then~$\rho'(\chi_1)$ is the permutation matrix~$P$
such that for each standard basis vector~$e_1,\dots,e_n$, we have~$Pe_i=e_{i-1}$
(with subscripts modulo~$n$).  Let~$\Gamma=\im(\rho)$ and~$\Gamma':=\im(\rho')$,
and define indeterminates~$y=Ax$.  For~$f(y)\in\polyringy$,
\[
  (f\circ P)(y_1,\dots,y_n)=f(y_2,y_3,\dots,y_{n},y_1).  
\] 
It follows that for each~$[D_j]\in\Jac(G)$ (considered as a character on~$\Jac(G)^*$), there is
  an induced, degree-preserving, linear isomorphism of relative invariant rings
  \begin{align}\label{eqn: permutation iso}
  \polyringy^{\Gamma'}_{[D_j]}&\xrightarrow{\sim}\polyring^{\Gamma}_{[D_j]}\\\nonumber
  f\quad&\mapsto f\circ A.
\end{align}

For each $N\in\neck(n,k)$, fix a representative word~$a_N$ for the
necklace~$\code(N)$, and use the Reynolds operator to define \[
f_N(y):=\mathcal{R}_{[D_j]}(y^{a_N})=\frac{1}{n}\sum_{i=1}^n\omega^{-ij}y^{\sigma^i(a_N)},
\] where~$\sigma$ is the cyclic shift operator defined earlier
and~$y^{a_N}:=\prod_{i=1}^ny_i^{a_{N,i}}$.  

\begin{thm}\label{thm: permutation invariants}  Let~$j\in[n]$ and define
\[
  m_j:=\frac{n+k}{\gcd(n,k,j)}\quad\text{and}\quad
  n_j:=\frac{n}{\gcd(n,j)}=\mathrm{order}(\omega^j).
\]
\begin{enumerate}[label=\rm{(\arabic*)},leftmargin=*]
  \item $N\in\neck(n,k)$ is~$m_j$-divisible if and only if~$\code(N)$
    is~$n_j$-divisible.
  \item\label{basis part2}  The set $\left\{ f_N \right\}$
    as~$N$ ranges over all~$m_j$-divisible~$N\in\neck(n,k)$ is a basis for
    the~$[D_j]$-relative invariants of degree~$k$ for the permutation
    representation~$\rho'$.
\end{enumerate}
\end{thm}

\begin{proof}
Let~$N\in\neck(n,k)$, and let~$a=a_N=a_1\dots a_n$ be a representative word
for $\code(N)$.  We first show that~$N$ is~$m_j$-divisible if and only
if~$\code(N)$ is~$n_j$-divisible.  Let~$\ell$ be the period of~$\code(N)$.  Then since
the length of the necklace~$\code(N)$ is~$n$, there is an integer~$p$ such
that~$n=p\ell$.  The period of~$N$ is~$\ell+\alpha$
where~$\alpha:=\sum_{i=1}^{\ell}a_i$.  Since~$\sum_{i=1}^{n}a_i=k$, and the
period of~$\code(N)$ is~$\ell$, it follows that~$p\alpha=k$.  Therefore,
\begin{align*}
  \text{$N$ is~$m_j$-divisible}&\quad\Longleftrightarrow\quad
  \left( \frac{n+k}{\gcd(n,k,j)}\right)\bigg|\,(\ell+\alpha)\\[8pt]
  &\quad\Longleftrightarrow\quad\frac{p(\ell+\alpha)}{\gcd(p\ell,p\alpha,j)}\bigg|\,(\ell+\alpha)\\[8pt]
   &\quad\Longleftrightarrow\quad \frac{\gcd(p\ell,p\alpha,j)}{p(\ell+\alpha)}\cdot(\ell+\alpha)\in\Z\\[8pt]
   &\quad\Longleftrightarrow\quad p|j.
\end{align*}
Similarly,
\[
  \text{$\code(N)$ is~$n_j$-divisible}
  \quad \Leftrightarrow\quad\left( \frac{n}{\gcd(n,j)} \right)\bigg|\,\ell
  \quad\Leftrightarrow\quad\frac{p\ell}{\gcd(p\ell,j)}\bigg|\,\ell
  \quad\Leftrightarrow\quad p|j.
\]

Continuing with the notation already established, we now prove part~\ref{basis
part2}.  Since~$f_N$ is defined using the Reynolds operator, it
is~$[D_j]$-invariant. To see that it is non-zero, we show that the
monomial~$y^a$ appears in~$f_N$ with a nonzero coefficient. We
have~$y^a=y^{\sigma^i(a)}$ if and only if~$i$ is a multiple of the
period~$\ell$.  However,~$\ell$ is divisible by~$n_j$, which is the order
of~$\omega^{j}$.  So the coefficient of~$y^{\sigma^i(a)}$ in the expression for
the Reynolds operator is~$\omega^{-ij}/n=1/n$.  Therefore, the coefficient
of~$y^{a}$ in~$f_N$ is the integer~$p/n=1/\ell\neq0$.  It now follows that~$f_N$
has degree~$k$.

Let~$\mathcal{B}$ be the set of~$f_N$ as~$N$ varies over~$m_j$-divisible
elements of~$\neck(n,k)$.  Distinct elements of~$\mathcal{B}$ share no monomials
in common, and hence~$\mathcal{B}$ is a linearly independent set.  To
show~$\mathcal{B}$ spans the relative invariant module and finish the
proof, let~$h\in\polyringy^{\Gamma'}_{[D_j],k}$ be a
homogeneous~$[D_j]$-invariant of degree~$k$. For the sake of contradiction,
suppose~$h\notin\Span(\mathcal{B})$.  We regard~$h$ as a sum of terms where each
term is a nonzero constant times a monomial, and the monomials are distinct.
Among all elements of~$\polyringy^{\Gamma'}_{[D_j],k}$ that are
not in~$\Span(\mathcal{B})$, let~$h$ be one with the fewest number of terms, and
let~~$\beta y^b$ be one of these terms. Then~$\beta\mathcal{R}_{[D_j]}(y^b)$ is a
sum of terms appearing in~$h$.  Let~$N\in\neck(n,k)$ be the necklace
with~$\code(N)=[b]$.  Say~$[b]$ has period~$m$ and write~$n=mq$ for some
integer~$q$.  The coefficient of~$y^b$ in~$\mathcal{R}_{[D_j]}(y^b)$ is
\[
  \frac{1}{n}\sum_{i=1}^q(\omega^{-jm})^i,
\]
and the order of~$\omega^{-jm}$ is~$q/\gcd(q,j)$, a divisor of~$q$.
If~$[b]$ is not~$n_j$-divisible, then the order of~$\omega^j$ does not
divide~$m$, and hence,~$\omega^{jm}\neq1$.  It would then follow that the above sum
is~$0$, which contradicts the fact that~$\beta y^b$ is a term of~$h$.
Therefore,~$[b]$ is~$n_j$-divisible
and~$\mathcal{R}_{[D_j]}(y^b)=f_N\in\mathcal{B}$.  However, then the polynomial~$h-\beta f_N$
is an element of~$\polyringy^{\Gamma'}_{[D_j],k}$ with fewer terms
than~$h$ and not in~$\Span(\mathcal{B})$, which is a contradiction.
So~$\Span(\mathcal{B})=\polyringy^{\Gamma'}_{[D_j],k}$.
\end{proof}

\begin{cor}\label{cor: permutation invariants} With notation as in the theorem,
\begin{align*}
  \#|D_j+kq|&=  \#\left\{ N\in\neck(n,k):\text{$N$ is~$m_j$-divisible} \right\}\\[5pt]
	    &=  \#\left\{ N\in\neck(n,k):\text{$\code(N)$ is~$n_j$-divisible}
	    \right\}.
\end{align*}
In particular,~$\#|kq|=\#\neck(n,k)$.
\end{cor}

\begin{proof} The result follows immediately from Corollary~\ref{cor:
  lambda-molien}, Theorem~\ref{thm: permutation invariants}, and the
degree-preserving isomorphism~(\ref{eqn: permutation iso}).
\end{proof}

\begin{example} Consider the case~$(n,k,j)=(4,2,1)$.  We have
\[
  m_1=\frac{4+2}{\gcd(4,2,1)}=6\quad\text{and}\quad n_1=\frac{4}{\gcd(4,1)}=4.
\]
The three elements in~$\neck(4,2)$ are pictured in Figure~\ref{fig: necklaces}, and
have codes~$[2000]$,~$[1100]$, and~$[1010]$.  Of these, only, the first two
are~$n_1$-divisible.
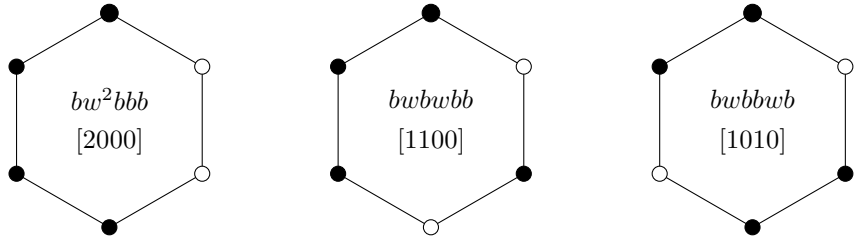
\begin{figure}[ht!]
\centering
  \def\r{1.5}
  \def\xshift{1.2in}
  \def\yshift{-1.0in}
  \def\rot{90}
  \begin{center}
    \begin{tikzpicture}[scale=0.95]
      \draw (0+\rot:\r) \foreach \x in {0,1,2,3,4,5} {
	      -- (\x*360/6+\rot:\r)
	  } -- cycle (0+\rot:\r);
	  \draw[fill=black] (0+\rot:\r) circle(3.5pt);
      \draw[fill=black] (360/6+\rot:\r) circle(3pt);
      \draw[fill=black] (2*360/6+\rot:\r) circle(3pt);
      \draw[fill=black] (3*360/6+\rot:\r) circle(3pt);
      \draw[fill=white] (4*360/6+\rot:\r) circle(3pt);
      \draw[fill=white] (5*360/6+\rot:\r) circle(3pt);
      \node at (0,0.3) {$bw^2bbb$};
      \node at (0,-0.3) {$[2000]$};

      \begin{scope}[xshift=4.5cm]
      \draw (0+\rot:\r) \foreach \x in {0,1,2,3,4,5} {
	      -- (\x*360/6+\rot:\r)
	  } -- cycle (0+\rot:\r);
	  \draw[fill=black] (0+\rot:\r) circle(3.5pt);
      \draw[fill=black] (360/6+\rot:\r) circle(3pt);
      \draw[fill=black] (2*360/6+\rot:\r) circle(3pt);
      \draw[fill=white] (3*360/6+\rot:\r) circle(3pt);
      \draw[fill=black] (4*360/6+\rot:\r) circle(3pt);
      \draw[fill=white] (5*360/6+\rot:\r) circle(3pt);
      \node at (0,0.3) {$bwbwbb$};
      \node at (0,-0.3) {$[1100]$};
    \end{scope}
      \begin{scope}[xshift=9cm]
      \draw (0+\rot:\r) \foreach \x in {0,1,2,3,4,5} {
	      -- (\x*360/6+\rot:\r)
	  } -- cycle (0+\rot:\r);
	  \draw[fill=black] (0+\rot:\r) circle(3.5pt);
      \draw[fill=black] (360/6+\rot:\r) circle(3pt);
      \draw[fill=white] (2*360/6+\rot:\r) circle(3pt);
      \draw[fill=black] (3*360/6+\rot:\r) circle(3pt);
      \draw[fill=black] (4*360/6+\rot:\r) circle(3pt);
      \draw[fill=white] (5*360/6+\rot:\r) circle(3pt);
      \node at (0,0.3) {$bwbbwb$};
      \node at (0,-0.3) {$[1010]$};
    \end{scope}
    \end{tikzpicture}
  \end{center}
  \caption{Binary necklaces with~$4$ black beads and~$2$ white beads with their
  codes.  The first two necklaces are~$1$-,~$2$-,~$3$- and~$6$-divisible, and the last one
is~$1$- and~$3$-divisible.}\label{fig: necklaces}
\end{figure}

\noindent As an instance of Theorem~\ref{thm: permutation invariants}, apply the
Reynold's operator with~$\omega=\I=\sqrt{-1}$ to the two~$6$-divisible necklaces
to find a basis for~$\polyringy_{[D_1],2}^{\Gamma'}$:
\begin{align*}
  4f_{[2000]} &= \I^{-1}y_2^2+\I^{-2} y_3^{2}+\I^{-3}y_4^2+\I^{-4}y_1^2=
  -\I y_2^2-y_3^2+\I y_4^2+ y_1^2\\[5pt]
  4f_{[1100]}&=\I^{-1}y_2y_3+\I^{-2}y_3y_4+\I^{-3}y_4y_1+\I^{-4}y_1y_2
  =-\I y_2y_3-y_3y_4+\I y_1y_4+y_1y_2.
\end{align*}
To change to the basis corresponding to the diagonal representation, substitute
\begin{align*}
  y_1 &= x_1+x_2+x_3+x_4\\
  y_2 &= \I x_1- x_2-\I x_3 +x_4\\
  y_3 &= -x_1+ x_2-x_3+x_4\\
  y_4 &= -\I x_1- x_2+\I x_3+x_4,
\end{align*}
to find
\begin{align*}
  f_{[2000]}&= 2(x_{1} x_{4} + x_{2} x_{3})\\
  f_{[1100]}&= \left(1+\I\right)(x_{1} x_{4} -  x_{2} x_{3}),
\end{align*}
which is a basis for~$\polyring_{[D_1],2}^{\Gamma}$.  Note that~$x_1x_4$
and~$x_2x_3$ form a monomial basis for~$\polyring_{[D_1],2}^{\Gamma}$ whose
exponent vectors are exactly the elements
of~$|D_1+2q|=\left\{(1,0,0,1),(0,1,1,0) \right\}$ in accordance with
Theorem~\ref{thm: invariants}. 

Next, consider the case~$(n,k,j)=(4,2,2)$.  We have~$m_2=3$ and~$n_2=2$.  All
three necklaces in~$\neck(4,2)$ are~$3$-divisible.  The corresponding basis
for~$\polyringy_{[D_2],2}^{\Gamma'}$ is
\begin{align*}
  4f_{[2000]} &= \I^{-2}y_2^2+\I^{-4} y_3^{2}+\I^{-6}y_4^2+y_1^2=
  -y_2^2+y_3^2- y_4^2+ y_1^2\\[5pt]
  4f_{[1100]}&=\I^{-2}y_2y_3+\I^{-4}y_3y_4+\I^{-6}y_4y_1+y_1y_2
  =- y_2y_3+y_3y_4- y_1y_4+y_1y_2\\[5pt]
  4f_{[1010]}&=\I^{-2}y_2y_4+\I^{-4}y_3y_1+\I^{-6}y_4y_2+y_1y_3
  =2y_1y_3- 2y_2y_4.
\end{align*}
Substitute to get a basis for~$\polyring_{[D_2],2}^{\Gamma}$:
\begin{align*}
  f_{[2000]}&=x_{1}^{2}+ 2 \, x_{2} x_{4}+  x_{3}^{2}  \\
  f_{[1100]}&= \I (x_{1}^{2} -  x_{3}^{2})\\
  f_{[1010]}&= -x_{1}^{2} + 2 \, x_{2} x_{4}- x_{3}^{2}.
\end{align*}
The corresponding complete linear system
is~$|D_2+2q|=\left\{(2,0,0,0),(0,1,0,1),(0,0,2,0) \right\}$, which by
Theorem~\ref{thm: invariants} yields the monomial
basis~$\left\{x_1^2,x_2x_4,x_3^2\right\}$ for~$\polyring^{\Gamma}_{[D_2],2}$ .
\end{example}

\subsection{Combinatorial bijection}  We now give an independent proof of
Corollary~\ref{cor: permutation invariants} in the case where~$n$ and~$k$ are
relatively prime.
Given~$D\in\Div(C_n)\simeq\Z^n$ and~$v\in\Z^n$, let~$D\cdot v$ be the usual dot
product of vectors.  Extend the rotation operator~$\sigma$ on words to
divisors by letting~$\sigma(D)(v_i):=D(v_{i+1})$ for all~$i$ modulo~$n$.

\begin{lemma}\label{lemma: divisors on C_n}  Let~$n,k\in\Z$ with~$n>1$, and let~$D\in\Div(C_n)$.
  Let
  \[
    \eta := (1,2,\dots,n)\quad\text{and}\quad\vec{1}=(1,1,\dots,1)\in\Z^n.
  \]
  Then~$D\sim D_j+kq$ if and only~$D\cdot\vec{1}=k$ and~$D\cdot\eta=j\bmod n$.
\end{lemma}

\begin{proof} First note that if~$c$ is a column of the Laplacian matrix
for~$C_n$, then~$c\cdot\eta=0\bmod n$.  Given any~$D\in\Div(C_n)$, there
exists~$i$ and~$m$ such that~$D\sim D_i+mq$.  Then~$D\cdot\vec{1}=m$, and
\[
  D\cdot\eta = (D_i+mq)\cdot\eta\bmod n= i\bmod n.
\]
The result follows.
\end{proof}

\begin{thm}\label{thm: necklace bijection}
  Given an effective divisor~$E=(E(v_1),\dots,E(v_n))\in\N^n\subset\Div^k(C_n)$, define the
  word 
\[
  w_E:=bw^{E(v_1)}bw^{E(v_2)}\dots bw^{E(v_n)},
\]
and the corresponding necklace~$N_E:=[w_E]\in\neck(n,k)$ with~$\code(N_E)=[E(v_1)\dots
E(v_n)]$.  If $\gcd(n,k)=1$, then for each~$j\in[n]$,
\begin{align*}
 \psi\colon |D_j+kq|&\to\neck(n,k)\\
  E&\mapsto N_E
\end{align*}
is a bijection.
\end{thm}

\begin{proof} Let~$\eta=(1,2,\dots,n)$ as in Lemma~\ref{lemma: divisors on C_n}.
  To show injectivity, suppose~$\psi(E)=\psi(E')$ for some pair of
  effective divisors~$E,E'\in\Div^k(C_n)$.  It follows that~$E'=\sigma^i(E)$ for
  some~$i$.  By Lemma~\ref{lemma: divisors on C_n}, working modulo~$n$,
  \[
    j = E'\cdot\eta =\sigma^i(E)\cdot\eta=E\cdot(\eta+i\vec{1})=j+ik\bmod n.
  \]
  If~$\gcd(n,k)=1$, it follows that~$i=0\bmod n$, and hence~$E=E'$.

  For surjectivity, let~$N\in\neck(n,k)$ with~$\code(N)=[a_1\dots a_n]$.
  Let~$E:=(a_1,\dots,a_n)\in\Div^k(C_n)$, and say~$E\cdot\eta=m\bmod n$.
  Then $\sigma^i(E)\cdot\eta=m+ik\bmod n$ for each~$i$.  If~$\gcd(n,k)=1$,
  we can take~$i$ so that~$m+ik=j\bmod n$ and define~$E':=\sigma^i(E)$.
  Then~$E'\in|D_j+kq|$ and~$\psi(E')=N$.
\end{proof}

\begin{remark}  If~$N\in\neck(n,k)$ and~$\code(N)$ has period~$\ell$, then as we
  saw in the proof of Theorem~\ref{thm: permutation invariants}, both~$\ell|n$
  and~$\ell|k$.  Thus, if~$\gcd(n,k)=1$, it follows that~$\ell=1$.  In other
  words, each element of~$\neck(n,k)$ has period~$n+k$.  Further, by the
  proof of Theorem~\ref{thm: necklace bijection}, there is a commutative
  diagram of isomorphisms of sets:
  \begin{center}
    \begin{tikzcd} 
      {|kq|}\arrow[d,"\sigma^j"',pos=0.35]\arrow[r,"\psi"] & \neck(n,k)\\
      {|D_j+kq|}\arrow[ru,"\psi"',pos=0.45]&
  \end{tikzcd}.
  \end{center}
\end{remark}

\begin{remark}  (Duality) Switching colors gives a bijection
  between~$\neck(n,k)$ and~$\neck(k,n)$.  Therefore, fixing vertices~$q$
  on~$C_n$ and~$q'$ on~$C_k$, Corollary~\ref{cor: permutation invariants} says
  that the cardinality of $|kq|$ on~$C_n$ is equal to that of ~$|nq'|$ on~$C_k$.  Further, when~$n$
  and~$k$ are relatively prime, Theorem~\ref{thm: necklace bijection} gives a
  combinatorial bijection between these complete linear systems.
\end{remark}

\begin{example} Figure~\ref{fig: C_3 necklaces} illustrates the bijection of
  Theorem~\ref{thm: necklace bijection} for the case~$n=3$ and~$k=4$ and for
  all~$j=0,1,2$.  The linear systems~$|D_j+4q|$ are the same up to cyclic rotation:
  \begin{align*}
    |4q|&=\left\{ (0,0,4),(0,3,1),(1,1,2),(2,2,0),(3,0,1) \right\}\\
    |D_1+4q|&=\left\{ (4,0,0),(1,0,3),(2,1,1),(0,2,2),(1,3,0) \right\}\\
    |D_2+4q|&=\left\{ (0,4,0),(3,1,0),(1,2,1),(2,0,2),(0,1,3) \right\}.
  \end{align*}
\begin{figure}[ht!]
\centering
  \def\r{1.1}
  \def\xshift{1.2in}
  \def\yshift{-1.0in}
  \def\rot{3*360/28}
  \begin{center}
    \begin{tikzpicture}[scale=0.95]
      \draw (0+\rot:\r) \foreach \x in {0,1,...,6} {
	      -- (\x*360/7+\rot:\r)
	  } -- cycle (0+\rot:\r);
      \draw[fill=black] (0+\rot:\r) circle(3pt);
      \draw[fill=black] (360/7+\rot:\r) circle(3pt);
      \draw[fill=black] (2*360/7+\rot:\r) circle(3pt);
      \draw[fill=white] (3*360/7+\rot:\r) circle(3pt);
      \draw[fill=white] (4*360/7+\rot:\r) circle(3pt);
      \draw[fill=white] (5*360/7+\rot:\r) circle(3pt);
      \draw[fill=white] (6*360/7+\rot:\r) circle(3pt);
      \begin{scope}[yshift=2.8]
      \draw (270:\r/2.1) node[below] {$\scriptstyle 4$}
      --(30:\r/2.1) node[right] {$\scriptstyle 0$}
      --(150:\r/2.1) node[left] {$\scriptstyle 0$}
	  --cycle;
      \end{scope}
      \begin{scope}[xshift=\xshift]
	\draw (0+\rot:\r) \foreach \x in {0,1,...,6} {
		-- (\x*360/7+\rot:\r)
	    } -- cycle (0+\rot:\r);
	\draw[fill=black] (0+\rot:\r) circle(3pt);
	\draw[fill=white] (360/7+\rot:\r) circle(3pt);
	\draw[fill=white] (2*360/7+\rot:\r) circle(3pt);
	\draw[fill=white] (3*360/7+\rot:\r) circle(3pt);
	\draw[fill=black] (4*360/7+\rot:\r) circle(3pt);
	\draw[fill=white] (5*360/7+\rot:\r) circle(3pt);
	\draw[fill=black] (6*360/7+\rot:\r) circle(3pt);
        \begin{scope}[yshift=2.8]
	\draw (270:\r/2.1) node[below] {$\scriptstyle 1$}
	--(30:\r/2.1) node[right] {$\scriptstyle 0$}
	--(150:\r/2.1) node[left] {$\scriptstyle 3$}
	    --cycle;
	  \end{scope}
      \end{scope}
      \begin{scope}[xshift=2*\xshift]
	\draw (0+\rot:\r) \foreach \x in {0,1,...,6} {
		-- (\x*360/7+\rot:\r)
	    } -- cycle (0+\rot:\r);
	\draw[fill=white] (0+\rot:\r) circle(3pt);
	\draw[fill=black] (360/7+\rot:\r) circle(3pt);
	\draw[fill=white] (2*360/7+\rot:\r) circle(3pt);
	\draw[fill=black] (3*360/7+\rot:\r) circle(3pt);
	\draw[fill=white] (4*360/7+\rot:\r) circle(3pt);
	\draw[fill=white] (5*360/7+\rot:\r) circle(3pt);
	\draw[fill=black] (6*360/7+\rot:\r) circle(3pt);
        \begin{scope}[yshift=2.8]
	\draw (270:\r/2.1) node[below] {$\scriptstyle 2$}
	--(30:\r/2.1) node[right] {$\scriptstyle 1$}
	--(150:\r/2.1) node[left] {$\scriptstyle 1$}
	    --cycle;
	  \end{scope}
      \end{scope}
      \begin{scope}[xshift=3*\xshift]
	\draw (0+\rot:\r) \foreach \x in {0,1,...,6} {
		-- (\x*360/7+\rot:\r)
	    } -- cycle (0+\rot:\r);
	\draw[fill=white] (0+\rot:\r) circle(3pt);
	\draw[fill=black] (360/7+\rot:\r) circle(3pt);
	\draw[fill=white] (2*360/7+\rot:\r) circle(3pt);
	\draw[fill=white] (3*360/7+\rot:\r) circle(3pt);
	\draw[fill=black] (4*360/7+\rot:\r) circle(3pt);
	\draw[fill=black] (5*360/7+\rot:\r) circle(3pt);
	\draw[fill=white] (6*360/7+\rot:\r) circle(3pt);
        \begin{scope}[yshift=2.8]
	\draw (270:\r/2.1) node[below] {$\scriptstyle 0$}
	--(30:\r/2.1) node[right] {$\scriptstyle 2$}
	--(150:\r/2.1) node[left] {$\scriptstyle 2$}
	    --cycle;
	  \end{scope}
      \end{scope}
      \begin{scope}[xshift=4*\xshift]
	\draw (0+\rot:\r) \foreach \x in {0,1,...,6} {
		-- (\x*360/7+\rot:\r)
	    } -- cycle (0+\rot:\r);
	\draw[fill=white] (0+\rot:\r) circle(3pt);
	\draw[fill=white] (360/7+\rot:\r) circle(3pt);
	\draw[fill=white] (2*360/7+\rot:\r) circle(3pt);
	\draw[fill=black] (3*360/7+\rot:\r) circle(3pt);
	\draw[fill=black] (4*360/7+\rot:\r) circle(3pt);
	\draw[fill=white] (5*360/7+\rot:\r) circle(3pt);
	\draw[fill=black] (6*360/7+\rot:\r) circle(3pt);
        \begin{scope}[yshift=2.8]
	\draw (270:\r/2.1) node[below] {$\scriptstyle 1$}
	--(30:\r/2.1) node[right] {$\scriptstyle 3$}
	--(150:\r/2.1) node[left] {$\scriptstyle 0$}
	    --cycle;
	\end{scope}
      \end{scope}
    \end{tikzpicture}
  \end{center}
  \caption{The complete linear system~$|D_j+4q|$ on~$C_3$ 
  in bijection with the necklaces~$\neck(3,4)$ according to Theorem~\ref{thm:
necklace bijection}.}\label{fig: C_3 necklaces} 
\end{figure}
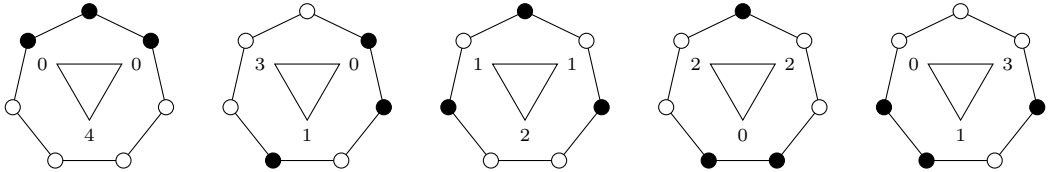 
\end{example}

\section{Extension to~$M$-matrices}\label{sect: m-matrices}  In this section,
we explain how to extend our results to a broader context.  It has been shown that
many aspects of the divisor theory of graphs are retained by a theory in which
reduced Laplacians of graphs are replaced by the more general class of matrices
called~{\em ~$M$-matrices} (\cite{Gabrielov}, \cite{Guzman}, \cite{Postnikov}).
To establish notation: if~$H$ and~$H'$ are matrices or vectors of the same
dimensions, write~$H\geq H'$ (resp.,~$H>H'$) if each entry of~$H-H'$ is
nonnegative (resp., positive).
\begin{defn}  Let~$A$ be an~$(n-1)\times(n-1)$ matrix over~$\R$ with~$A_{ij}\leq
  0$ for all~$i\neq j$.  Then~$A$ is a (non-singular)~{\em $M$-matrix} if any of
  the following equivalent conditions holds:
  \begin{enumerate}
    \item $A=sI_{n-1}-B$ for some matrix~$B\geq0$ and some 
      $s>\max\left\{ |\lambda|:\text{$\lambda$ an eigenvalue of~$B$}\right\}$.
    \item Each eigenvalue of~$A$ has positive real part.
    \item Each principal minor of~$A$ is positive.
    \item $A^{-1}$ exists and~$A^{-1}\geq0$.
    \item If~$Au\geq0$, then~$u\geq0$.
    \item There exists~$u>0$ such that~$Au>0$.
    \item\label{item: burning} There exists~$u>0$ with~$Au\geq0$ and such that
      if~$(Au)_{i_0}=0$ for some~$i_0$, 
      then there exists indices~$i_1,\dots,i_r$ with~$A_{i_{k}i_{k+1}}\neq0$
      for~$k=1,\dots,r-1$ and~$(Au)_{i_r}>0$.
  \end{enumerate}
\end{defn}
The above seven equivalent conditions come from the list of~$40$ given
by Plemmons (\cite{Plemmons}).  

From now on, we assume that~$A$ is an integer~$M$-matrix.  In that case, any
integer vector~$u$ satisfying property~(\ref{item: burning}) is called a~{\em
burning script}.  A burning script for~$A$ always exists and a unique minimal
one (with respect to~$\leq$) can be constructed as follows: start
with~$u=(1,\dots,1)$, and then as long as~$(Au)_i<0$ for some~$i$,
increase~$u_i$ by~$1$ (\cite{PPW}; \cite{Corry}, Chapter~7). If~$u$ is a burning
script, then~$Au$ is called a {\em burning configuration}.

Let~$u=(u_1,\dots,u_n)>0$ and~$w=(w_1,\dots w_n)>0$ be any integer vectors such that both~$Au\geq0$
and~$wA\geq0$. Their existence is guaranteed by property~(\ref{item: burning}) and
the fact that the transpose~$A^t$ of~$A$ is also an~$M$-matrix. We
do not require that~$u$ and~$w$ be burning scripts.  Next, define the {\em
$(w,u)$-extension} of~$A$ to be the~$n\times n$ matrix
\[
  \wA:=
  \left(\begin{array}{c|c}
      wAu&-wA\\\hline
      -Au&A
  \end{array} \right).
\]
The following vectors are primitive generators for the left and right kernels,
respectively, of~$\wA$:
\[
  \phi=(1,w_1,w_2,\dots,w_{n-1})
  \quad\text{and}\quad
  \delta=(1,u_1,u_2,\dots,u_{n-1}).
\]

\begin{example} Let~$L$ be the Laplacian matrix for a connected, undirected
  graph with respect to some ordering of the vertices, and let~$A=\tL$ be the
  corresponding reduced Laplacian with respect to the first vertex.
  Then~$A=A^t$ is an~$M$-matrix (\cite{Guzman}) with minimal
  burning script~$u=w=(1,\dots,1)$.  The~$(w,u)$-extension of~$A$
  recovers~$L$, i.e.,~$\wA=L$. 
\end{example}

We now extend our earlier results on the cardinality of complete linear systems
to the setting of~$M$-matrices.  A {\em divisor} is an
element~$D\in\Div(\wA):=\Z^n$.  The {\em degree} of a divisor~$D$ is given by
the dot product~$\deg(D):=\phi\cdot D$.  Define linear equivalence of divisors
by~$D\sim D'$ if~$D-D'\in\im_{\Z}\wA$.  As before,
let~$\Pic(\wA):=\Z^{n}/\!\sim$, which is graded by (our new) degree,
and~$\Jac(\wA):=\Pic^{0}(\wA)$, the group of divisor classes of divisors of
degree~$0$.  


For notational purposes, define~$v_i:=e_i$, the~$i$-th standard basis vector,
for ~$i=1,\dots,n$. The isomorphisms~\eqref{eqn:
Pic(G)} and~\eqref{eqn: Jac(G)} of Section~\ref{sect: prelims} hold in this new
setting in which~$\tL$ is replaced by~$A$.  For each divisor
class~$[D]\in\Pic(\wA)$, define the {\em complete linear system}~$|D|$, the
set~$\E_{[D]}$, and the~$\lambda$-generating function~$\Lambda_{[D]}(z)$ as in
Section~\ref{sect: prelims}.  Substituting~$\wA$ and~$A$ for the Laplacian and
reduced Laplacian, respectively, our main results generalize, with nearly
identical proofs, after suitably modifying the statements to take into account
our new notion of degree:
\smallskip

\noindent{\em Primary and secondary divisors.}  For each~$i=1,\dots,n$,  the
degree of~$v_i$ considered as a divisor is~$\deg(v_i)=\phi_i$.
Redefine~$\ord_q(v_i)$ to be the order of~$[v_i-\deg(v_i)q]\in\Pic(\wA)$, and
let~$\ell_i$ be any positive integer multiple of~$\ord_q(v_i)$.  Then
Theorem~\ref{thm: primary-secondary}, Corollary~\ref{cor: ps-genfun}, and
Proposition~\ref{prop: ps-computation} hold after replacing each occurrence
of~$\ell_{v}$ with~$\phi_i\ell_i$.  For instance, in Corollary~\ref{cor:
ps-genfun}, we now have
\begin{equation}\label{eqn: ps-genfun}
  \Lambda_{[D]}(z)=\frac{S(z)}{\prod_{i=1}^{n}(1-z^{\phi_i\ell_i})}
\end{equation}
where~$S(z):=\sum_{F\in\secondary_{[D]}}z^{\deg(F)}$.
\smallskip

\noindent{\em Polyhedra.}  The constructions in Section~\ref{sect: poly} remain
valid. One cosmetic change in the exposition is that instead of taking~$q$ to be
the last vertex of the graph, we now take~$q$ to be the {\em first} standard
basis vector.\footnote{This switch in the placement of~$q$ was made in order to
  conform to the conventions for root systems considered in~\cite{BKR}.  See
Section~\ref{sect: roots}, below.}  This means, for example, that instead of
considering,~$(f,t)\in\R^{n-1}\times\R$, we now
consider~$(t,f)\in\R\times\R^{n-1}$.  Theorem~\ref{thm: poly} then
holds as stated, defining~$\deg(\omega):=\phi\cdot\omega$ in part~\ref{poly3}.
Proposition~\ref{prop: facet} holds by again replacing~$\ell_i$
by~$\phi_i\ell_i$ and redefining~$\ord_q(v_i)$ as discussed above.
\smallskip

\noindent{\em Invariant theory.}  To generalize the results in
Section~\ref{sect: molien}, take~$\polyring$ to have the multigrading determined
by~$\deg(x_i):=\phi_i$ for~$i=1,\dots,n$.  The representation of~$\rho$, described
in~(\ref{rho}), becomes
\begin{equation}\label{eqn: new rho}
  \rho(\chi)=(1,\chi([v_2-\deg(v_2)q]),\chi([v_3-\deg(v_3)q]),\dots,\chi([v_n-\deg(v_n)q])).
\end{equation}
Theorem~\ref{thm: invariants} then extends with no changes to its statement.  For
Corollary~\ref{cor: lambda-molien}, use a multigraded version of Molien's theorem,
for abelian groups, to get
\begin{equation}\label{eqn: new lambda_molien}
  \Lambda_{[D]}(z)=
\Phi_{\Gamma,\chi}(z^{\phi_1},\dots,z^{\phi_n})
:=\frac{1}{|\Jac(\wA)|}\sum_{\chi\in\Jac(\wA)^*}\frac{\overline{\chi([D])}}{\det(I_n-\diag(z^{\phi_1},\dots,z^{\phi_n})\rho(\chi))}
\end{equation}
where~$\diag(\,\cdot\,)$ denotes the diagonal matrix with the given diagonal entries.

\subsection{Root systems and McKay quivers.}\label{sect: roots and McKay}  
In (\cite{BKR}), Benkart, Klivans, and Reiner
relate two classes of~$M$-matrices to the extended divisor theory described
above.  For the sake of brevity, we give only a cursory description of some of
their work, referring the interested reader to the original paper for definitions and
other details.\footnote{Note that our convention for the Laplacian of a graph differs
from that in~\cite{BKR} by a transpose.}
\smallskip

\subsubsection{Root systems}\label{sect: roots}  Let~$\Phi$ be a finite,
crystallographic, irreducible root system. The Cartan matrix~$C$ for~$\Phi$ is
an~$M$-matrix. Its burning configurations are the elements of the root lattice
lying in the fundamental chamber (with respect to a choice of simple roots).
Making particular natural choices for burning configurations for~$C$ and its
transpose~$C^t$, the authors define the extended Cartan matrix~$\widetilde{C}$,
which is the Cartan matrix for the corresponding affine root system.
Letting~$A=C^t$ and~$\wA=\widetilde{C}^t$, it turns out that~$\Pic(\wA)$ is the
fundamental group of~$\Phi$, i.e., the quotient of the weight lattice by the
root lattice.  We think of each~$D\in\Div(\wA)$ as a divisor on the affine
Dynkin diagram for the affine root system corresponding to~$\Phi$, and the
matrix~$\wA$ can be thought of as defining firing rules (as described for the
Laplacian in Section~\ref{sect: prelims}).

\begin{example}  Let~$\Phi$ be the root system~$B_3$.  The transpose of its
  Cartan matrix is
    \[
A =
\left(\begin{array}{rrr}
    2&-1&0\\
    -1&2&-1\\
    0&-2&2
\end{array} \right)
\]
The vectors~$u=(1,2,2)$ and~$w=(1,2,1)$ are burning scripts for~$A$ and~$A^t$,
respectively, (though only the latter is minimal).  The $(w,u)$-extension
of~$A$ is then
\[
\wA = \widetilde{C}^t=
\left(\begin{array}{rrrr}
    2&0&-1&0\\
    0&2&-1&0\\
    -1&-1&2&-1\\
    0&0&-2&2
\end{array} \right),
\]
with left and right kernel generators:
\[
  \phi = (1,1,2,1)\quad\text{and}\quad\delta = (1,1,2,2).
\]
We have~$\Jac(\wA)\simeq\Z/2\Z$ with generator~$D:=(-1,0,0,1)$.  Follow the
procedure in Remark~\ref{remark: method} to compute primary and secondary
divisors:
\begin{align*}
  \primary\ &=  \left\{ (1,0,0,0),(0,1,0,0),(0,0,1,0),(0,0,0,2) \right\}\\
  \secondary_{[0]}&= \left\{ (0,0,0,0) \right\}\\
  \secondary_{[D]}&= \left\{ (0,0,0,1) \right\}.
\end{align*}
Note that the primary divisor~$(0,0,1,0)$ has degree~$\phi\cdot(0,0,1,0)=2$. 
By \eqref{eqn: ps-genfun},
\begin{align*}
  \Lambda_{[0]}(z)&= \frac{1}{(1-z)^2(1-z^2)^2}=
  1 + 2z + 5z^2 + 8z^3 + 14z^4 + 20z^5 + 30z^6 + 40z^7 + 55z^8
  +\cdots\\[8pt]
  \Lambda_{[D]}(z)&= \frac{z}{(1-z)^2(1-z^2)^2}=
  z + 2z^2 + 5z^3 + 8z^4 + 14z^5 + 20z^6 + 30z^7 + 40z^8 + 55z^9 +
  \cdots.
\end{align*}
There is one non-trivial character~$\chi$ for~$\Jac(\wA)$, determined
by~$\chi([D])=-1$.  For the modified representation~\eqref{eqn: new rho}
for~$\Jac(\wA)^*$, we have~$\rho(\chi)=(1,1,1,-1)$.  Therefore, the new Molien
series~\eqref{eqn: new lambda_molien} gives the following forms for
the~$\lambda$-generating functions:
\begin{align*}
  \Lambda_{[0]}(z)&= \frac{1}{2}\left( \frac{1}{(1-z)^2(1-z^2)(1-z)}
+\frac{1}{(1-z)^2(1-z^2)(1+z)}
  \right)\\[8pt]
  \Lambda_{[D]}&= \frac{1}{2}\left( \frac{1}{(1-z)^2(1-z^2)(1-z)}+
  \frac{-1}{(1-z)^2(1-z^2)(1+z)}\right).
\end{align*}
For example, the coefficient of~$z^2$ in the series expansion of~$\Lambda_{[0]}(z)$
indicates there are~$5$ effective divisors in the complete linear
system for the divisor~$2q=(2,0,0,0)$.  These are pictured in
Figure~\ref{fig: cartan}.
\begin{figure}[ht]
\centering
\begin{tikzpicture}[scale=0.9]
  \begin{scope}[scale=0.8]
    \node[circle,draw,inner sep=1.7pt,label={below:{$2$}}] at (-1,-0.86) (q) {};
    \node[circle,draw,inner sep=1.7pt,label={above:{$0$}}] at (-1,0.86) (v2) {};
    \node[circle,draw,inner sep=1.7pt,label={above:{$0$}}] at (0,0) (v3) {};
    \node[circle,draw,inner sep=1.7pt,label={above:{$0$}}] at (1.1,0) (v4) {};
    \draw (q)--(v3)--(v2);
    \draw[double distance=3pt] (v3)--(v4);
    \draw (0.4,-0.25)--(0.6,0)--(0.4,0.25);
  \end{scope}
  \begin{scope}[scale=0.8,xshift=4cm]
    \node[circle,draw,inner sep=1.7pt,label={below:{$0$}}] at (-1,-0.86) (q) {};
    \node[circle,draw,inner sep=1.7pt,label={above:{$0$}}] at (-1,0.86) (v2) {};
    \node[circle,draw,inner sep=1.7pt,label={above:{$1$}}] at (0,0) (v3) {};
    \node[circle,draw,inner sep=1.7pt,label={above:{$0$}}] at (1.1,0) (v4) {};
    \draw (q)--(v3)--(v2);
    \draw[double distance=3pt] (v3)--(v4);
    \draw (0.4,-0.25)--(0.6,0)--(0.4,0.25);
  \end{scope}
  \begin{scope}[scale=0.8,xshift=8cm]
    \node[circle,draw,inner sep=1.7pt,label={below:{$1$}}] at (-1,-0.86) (q) {};
    \node[circle,draw,inner sep=1.7pt,label={above:{$1$}}] at (-1,0.86) (v2) {};
    \node[circle,draw,inner sep=1.7pt,label={above:{$0$}}] at (0,0) (v3) {};
    \node[circle,draw,inner sep=1.7pt,label={above:{$0$}}] at (1.1,0) (v4) {};
    \draw (q)--(v3)--(v2);
    \draw[double distance=3pt] (v3)--(v4);
    \draw (0.4,-0.25)--(0.6,0)--(0.4,0.25);
  \end{scope}
  \begin{scope}[scale=0.8,xshift=12cm]
    \node[circle,draw,inner sep=1.7pt,label={below:{$0$}}] at (-1,-0.86) (q) {};
    \node[circle,draw,inner sep=1.7pt,label={above:{$2$}}] at (-1,0.86) (v2) {};
    \node[circle,draw,inner sep=1.7pt,label={above:{$0$}}] at (0,0) (v3) {};
    \node[circle,draw,inner sep=1.7pt,label={above:{$0$}}] at (1.1,0) (v4) {};
    \draw (q)--(v3)--(v2);
    \draw[double distance=3pt] (v3)--(v4);
    \draw (0.4,-0.25)--(0.6,0)--(0.4,0.25);
  \end{scope}
  \begin{scope}[scale=0.8,xshift=16cm]
    \node[circle,draw,inner sep=1.7pt,label={below:{$0$}}] at (-1,-0.86) (q) {};
    \node[circle,draw,inner sep=1.7pt,label={above:{$0$}}] at (-1,0.86) (v2) {};
    \node[circle,draw,inner sep=1.7pt,label={above:{$0$}}] at (0,0) (v3) {};
    \node[circle,draw,inner sep=1.7pt,label={above:{$2$}}] at (1.1,0) (v4) {};
    \draw (q)--(v3)--(v2);
    \draw[double distance=3pt] (v3)--(v4);
    \draw (0.4,-0.25)--(0.6,0)--(0.4,0.25);
  \end{scope}
\end{tikzpicture}
\caption{The complete linear system of the divisor~$2q=(2,0,0,0)$ for the root
system~$B_3$.}\label{fig: cartan}
\end{figure}
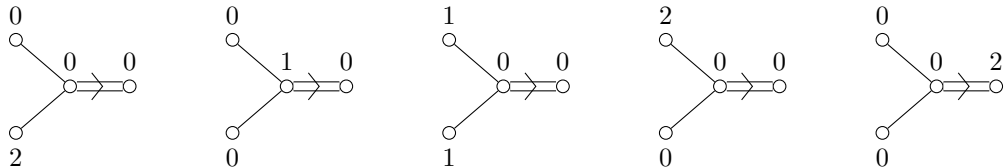
\end{example}

\subsubsection{McKay quivers}\footnote{For this section, in addition to~\cite{BKR}, see the work by
Gaetz,~\cite{Gaetz}.}
Let~$\rho\colon\Gamma\to\GL(\C^n)$ be a faithful representation of a finite
group~$\Gamma$ with character~$\chi_{\rho}$.
Let~$\rho_1,\dots,\rho_n$ be the irreducible complex representations
of~$\Gamma$, with~$\rho_1$ the trivial representation, and with respective
characters~$\chi_1,\dots,\chi_n$.  For each~$i$,
denote the character of the tensor product
~$\rho\otimes\rho_i$ by~$\chi_{\rho}\cdot\chi_i$, and define integers~$m_{ij}$ by
\[
  \chi_{\rho}\cdot\chi_i=\sum_{j=1}^nm_{ij}\chi_i.
\]
Define the~$n\times n$ matrix~$M:=(m_{ij})$ and the
{\em extended McKay-Cartan} matrix~$\widetilde{C}:=nI_n-M$.  The {\em McKay-Cartan}
matrix is then the submatrix~$C$ formed by removing the first row and first column
of~$\widetilde{C}$.  In our notation from above, take~$A=C^t$.  The vectors
$u=w=(\dim\rho_2,\dots,\dim\rho_n)$ are burning scripts for~$A$ and~$A^t$, with
respect to which~$\wA=\widetilde{C}^t$ with left and right kernel generators
\[
  \phi = \delta = (\dim\rho_1,\dots,\dim\rho_n).
\]
The {\em McKay quiver} of~$\gamma$ is the directed graph with
vertices~$\chi_1,\dots,\chi_n$ and~$m_{ij}$ directed edges from~$\chi_i$
to~$\chi_j$ for each~$i,j$.  The matrix~$\wA$ defines firing rules on the McKay
quiver (again as described for the Laplacian in Section~\ref{sect: prelims}).

\begin{example} Consider the representation~$\rho\colon\Jac(G)^*\to\GL(\C^n)$
  defined by~(\ref{rho}) of Section~\ref{sect: molien}.  When~$G=C_n$, the
  cyclic graph on~$n$-vertices,~$\rho$ is the regular representation of the
  cyclic group $\Jac(G)^*\simeq\Z/n\Z$.  Therefore,~$m_{ij}=1$ for all~$i,j$,
  the McKay quiver may be thought of as the (undirected) complete graph~$K_n$
  on~$n$ vertices, and $\wA$ is its Laplacian matrix.
  Thus,~$\Jac(\wA)=\Jac(K_n)$.

  More generally (\cite{BKR}, Section~6.2), the McKay quiver for any faithful complex
  representation~$\gamma$ of an abelian group has (directed) Laplacian matrix equal to the
  matrix~$\wA$ for~$\gamma$.
\end{example}

\section{Further work}\label{sect: further}
Here we suggest three possible directions for further inquiry.
\medskip

\noindent I. Let~$D$ be a divisor of degree~$k$ on a cycle graph with~$n$ vertices.
Corollary~\ref{cor: permutation invariants} shows that the complete linear
system $|D|$ can be enumerated using subsets of the set~$\neck(n,k)$ of binary
necklaces with~$n$ black beads and~$k$ white beads.  In particular,
if~$\gcd(n,k)=1$ or~$D=kq$, then~$\#|D|=\#\,\neck(n,k)$.  Theorem~\ref{thm:
necklace bijection} gives a combinatorial bijection between~$|D|$
and~$\neck(n,k)$ when~$\gcd(n,k)=1$. Motivated by this work, \cite{Oh} finds a
combinatorial bijection in the case~$D=kq$ when~$n$ is prime.  It would be
interesting to find combinatorial bijections for arbitrary~$n$ and~$k$.
\medskip

\noindent II. Section~\ref{sect: m-matrices} establishes tools for enumerating
linear systems related to root systems and representations of finite groups.
It may be worthwhile to investigate the implications for each type of root
system or for certain classes of representations.
\medskip

\noindent III. Assigning lengths to the edges of~$G$ results in a model for a
tropical curve~$\Gamma$ and an associated polytopal cell decomposition of
the~$k$-th symmetric power~$\mathrm{Sym}^k(\Gamma)$ for each~$k$.  If~$D$ is a
(tropical) divisor on~$\Gamma$ of degree~$k$, then the complete linear
system~$|D|$ can be realized as a cell complex in~$\mathrm{Sym}^k(\Gamma)$
(\cite{Gathmann}, \cite{Haase}, \cite{Mikhalkin}).  In the conclusion  of~\cite{Haase},
the authors suggest developing connections between their work on tropical linear
systems and the divisor theory for finite graphs. In particular, they ask for a
combinatorial description of each complete linear system on a finite graph
including a determination of its cardinality.  While we have determined this
cardinality, one could hope to further describe the combinatorics of the
complete linear system for a divisor on~$G$ within the cell complex for the
associated divisor on~$\Gamma$. Example~19 of~\cite{Haase} gives an explicit
description of these cell complexes for the case of the cyclic graph~$G=C_n$
whose linear systems we have related to binary necklaces.  That might be a good
place to start.  

\bibliographystyle{amsplain} 
\bibliography{els}
\end{document}